\documentclass[11pt]{article}
\usepackage{}
\usepackage{enumerate}
\usepackage{mathbbold}
\usepackage{bbm}
\usepackage{amsfonts}
\usepackage[T1]{fontenc}
\usepackage{latexsym,amssymb,amsmath,amsfonts,amsthm}
\usepackage{graphicx, color}
\usepackage{subfigure}
\usepackage{epsf}
\usepackage{epstopdf}
\usepackage{amssymb}
\usepackage{mathrsfs}
\usepackage[linesnumbered, ruled]{algorithm2e}
\usepackage{multirow}
\usepackage[ruled]{algorithm2e}

\renewcommand{\theequation}{\arabic{section}.\arabic{equation}}

\newcommand{\R}{{\mathbb R}}

\newcommand{\be}{\begin{eqnarray}}
\newcommand{\ben}{\begin{eqnarray*}}
\newcommand{\en}{\end{eqnarray}}
\newcommand{\enn}{\end{eqnarray*}}
\newcommand{\ba}{\backslash}
\newcommand{\pa}{\partial}

\newcommand{\Om}{\Omega}

\newcommand{\wi}{\widetilde}

\newcommand{\hx}{\hat{x}}
\SetKwRepeat{Do}{do}{while} 

\newtheorem{theorem}{Theorem}[section]
\newtheorem{lemma}[theorem]{Lemma}
\newtheorem{assumption}[theorem]{Assumption}
\begin{document}
\title{\bf The high resolution sampling methods for acoustic sources from multi-frequency far field patterns at sparse observation directions}
\author{
Xiaodong Liu\thanks{Academy of Mathematics and Systems Science,
Chinese Academy of Sciences, Beijing 100190, China. Email: xdliu@amt.ac.cn}
,\and
Qingxiang Shi\thanks{Corresponding author. Yau Mathematical Sciences Center, Tsinghua University,
Beijing 100084, China. Email: sqxsqx142857@tsinghua.edu.cn}
}
\date{}
\maketitle

\begin{abstract}
This work is dedicated to novel uniqueness results and high resolution sampling methods for source support from multi-frequency sparse far field patterns.
With a single pair of observation directions $\pm\hat{x}$, we prove that the lines $\{z\in\mathbb R^2|\, \hat{x}\cdot z = \hat{x}\cdot y, \,y\in A_{\hat{x}}\}$ can be determined by multi-frequency far field patterns at the directions $\pm\hat{x}$, where $A_{\hat{x}}$ denotes a set containing the corners of the boundary and points whose normal vector to the boundary is parallel to $\hat{x}$. Furthermore, if the source support is composed of polygons and annuluses, then we prove that the support can be determined by multi-frequency far field patterns at sparse directions.
Precisely, the lowest number of the observation directions is given in terms of the number of the corners and the annuluses.
Inspired by the uniqueness arguments, we introduce two novel indicators to determine the source support. Numerical examples in two dimensions are presented to show the validity and robustness of the two indicators for reconstructing the boundaries of the source support with a high resolution. The second indicator also shows its powerful ability to determine the unknown source function.

\vspace{.2in}

\vspace{.2in} {\bf AMS subject classifications:}
35R30, 78A46, 45Q05
\end{abstract}

\section{Introduction}
\setcounter{equation}{0}

 Inverse source problems (ISPs) play an important role in many scientific and industrial fields  such as medical imaging \cite{apply-at-2007,apply-at-2009}, antenna technology \cite{apply-at-1991,apply-at-1982}, and seismic monitoring \cite{apply-at-1986}. Due to the existence of non-radiating sources \cite{non-uni}, a general source cannot be uniquely determined from the far field patterns at a single frequency. 
 Therefore, in order to obtain a unique solution to the problem, one has to impose additional constraints on the source. A possible choice is to find the source with a minimum energy norm\cite{min-eng-sol-2007, min-eng-sol-2000, min-eng-sol-1982}. However, the minimum energy solution may not be the true source function. Moreover, the ISPs at fixed frequency are very difficulty to solve due to its inherited instability.  Therefore, there are many works for the ISPs with multi-frequency measurements. So far, there have been many results about uniqueness and stability on this aspect \cite{full-thm-2015,full-thm-2016,full-thm-2017,full-thm-2020} . Meanwhile, many numerical methods for ISPs have been proposed, for example, the Recursive algorithm \cite{full-num-recu}, the Fourier method \cite{full-num-fouri}, and the Learning method \cite{full-num-learn}. 
 
In many cases of practical interest, the measurements are only available at finitely many sensors. It is shown in \cite{JiLiu-sisc2020, JiLiu-siap2021}  that the point sources can be identified from the multi-frequency measurements at sparse sensors. Precisely, the lower bound of the needed sensors for uniquely determine the numbers, locations, and scattering strengths of the point sources is derived, a direct sampling method is proposed for locating all the locations and the formula for the scattering strengths is introduced. Moreover, these results have been extended to the case when the unknown point sources are distributed in a two-layered medium \cite{spar-point-2023} and the case of coexistence of point sources and extended sources \cite{spar-2023LiLiu, spar-2023}. However, the global uniqueness of a general extended source from the measurements at sparse sensors does not hold \cite{spar-2020}.  Consequently, what we can expect from the sparse measurements is partial information about the extended source, say e.g., the support of the source function. 

Mathematically, let $f: \mathbb R^2\to\mathbb C$ be a generic source function with compact support $\Om:=supp f$ located in an isotropic homogeneous medium. 
The source $f$ gives rise to a scattered field $u$ which satisfies 
\be
    \label{ela}
    \Delta u+k^2u=-f\quad{\rm in} \ \mathbb R^2,\\
    \label{Src}
    \lim\limits_{r=|x|\to\infty}r^{\frac{1}{2}}\left(\frac{\partial u}{\partial r} -i k u\right)=0,
\en
where $k>0$ is the wavenumber and the Sommerfeld radiation condition \eqref{Src} holds
uniformly with respect to $\hx=x/|x|\in  \mathbb S^{1}:=\{x\in\R^2:|x|=1\}$.
Every solution $u$ of \eqref{ela}-\eqref{Src} has the asymptotic behavior of an outgoing spherical wave
\be\label{0asyrep}
u(x,k)
=\frac{e^{i\frac{\pi}{4}}}{\sqrt{8k\pi}}
\frac{e^{ikr}}{r^{\frac{1}{2}}}\left\{u^{\infty}(\hat{x},k)+\mathcal{O}\left(\frac{1}{r}\right)\right\}\quad\mbox{as }\,r:=|x|\rightarrow\infty
\en
uniformly in all directions $\hx\in \mathbb S^{1}$ where the function $u^{\infty}$ defined on the unit circle $\mathbb S^1$ is known as the far field pattern of $u$. 
The radiating solution $u$ has the form
\be\label{S-rep}
    u(x,k)=\int_{\mathbb R^2}\Phi_k(x,y)f(y)dy, \quad  x\in\mathbb R^2,
\en
where
\ben
\Phi_k(x,y):=\frac{i}{4}H^{(1)}_0(k|x-y|),\quad x\neq y,
\enn
is the fundamental solution to the Helmholtz equation in $\R^2$ with $H^{(1)}_0$ denoting the Hankel function of the first kind and order zero.
Furthermore, the corresponding far field pattern $u^{\infty}$ takes the form
\be\label{u_inf_1}   
u^\infty(\hat{x},k)=\int_{\mathbb R^2}e^{-ik\hat{x}\cdot y}f(y)dy,\ \quad \hat{x}\in\mathbb S^{1}.
\en
In this paper, we are interested in the following inverse problem:\\
\textbf{\large{(IP)}:} Determine the support $\Omega$ of the source from the multi-frequency sparse far field patterns
\begin{align*}
    \{u^{\infty}(\hx,k)|\  \hx\in\Theta_L,\, 0< k_{\min}<k<k_{\max}\},
\end{align*}
 where 
 \ben
 \Theta_L:=\{\hat{x}_1,\hat{x}_2,\cdots,\hat{x}_L\}\subset\mathbb S^1
 \enn
 denotes the set of sparse observation directions. Here and throughout the paper we assume that any two elements in $\Theta_L$ are not collinear.

Under certain conditions, Sylvester and Kelly \cite{spar-2005} show that the strip 
\ben
S_{\Om}(\hx):=  \{ y\in \mathbb \R^{2}\; | \; \inf_{z\in \Om}z\cdot \hx \leq y\cdot \hx \leq \sup_{z\in \Om}z\cdot \hx\}
\enn
can be uniquely determined by the multi-frequency far field patterns at a fixed observation direction $\hx\in\Theta_L$.
Such a strip is the smallest strip with normals in the directions $\pm \hx$ that contains $\Om$.  Furthermore, one may obtain a union of convex polygons with normals in the observation directions $\hx\in\Theta_L$ containing the source support $\Om$. 
Numerically, since the pioneering work of Colton and Kirsch on the linear sampling method \cite{ColtonKirsch-lsm},  there has been an extensive study on the sampling type methods for support reconstructions in the last thirty years due to their many advantages, e.g., avoiding computing the direct problems, very simple and fast to implement, and making no use of the topological properties of $\Om$.
The first sampling type method for inverse source problems is a factorization method \cite{full-num-facto}, which has been proved to correctly recover the convex polygons containing the source support. The numerical examples show that the factorization method can be used to obtain a good approximation of the location and size of the source support $\Om$. More recently, some direct sampling methods have been proposed in \cite{spar-2020,JiLiu-siap2021,LiuMeng-csiam2023} to give a shape reconstruction of the source support $\Om$. In particular, numerical examples show that the concave part of $\Om$ can also be surprisingly recovered.  The main feature of the direct sampling methods is that only inner product of the
measurements with some suitably chosen functions is involved in the indicator functional and thus
seem very robust to noises. 

The first contribution of this paper is some uniqueness results of the source support $\Om$ from the  multi-frequency far field patterns at sparse observation directions.  Precisely,  if it is assumed that $\Om$ is composed of polygons and annuluses, we show the multi-frequency sparse far field patterns provide enough information to completely determine the  source support $\Om$. The low bound of needed number of the observation directions is given in terms of the number of annuluses and corners.
To the author's knowledge, this is the first uniqueness result for the source support $\Om$ from multi-frequency measurements at sparse sensors. 
Moreover, if we further assume that the source function $f$ is piecewise constant, then the source function $f$ can also be uniquely determined. 
An important ingredient in the uniqueness proof is an indicator for catching the corners and the circles, which can also be easily used for numerical computations. 

The other contribution of this paper is two novel direct sampling methods for reconstructing $\Om$ by using multi-frequency far field patterns at sparse observation directions. 
The first indicator function is motivated directly from the uniqueness arguments. The theoretical basis why it can be used to capture the polygons and annuluses has been established in the unique proof. The second indicator function is a slight modification of the first one. Generally speaking, the sampling type methods are qualitative methods in the sense that one can only reconstruct the support but not the parameters of the unknown objects. However, we show that the second indicator function can be used not only  to reconstruct the source support $\Om$ but also to determine the source function $f$ if there are sufficiently many observation directions. Therefore, the modification in the second indicator functional is nontrivial. 
Compared to the direct sampling method \cite{spar-2020,JiLiu-siap2021,LiuMeng-csiam2023} and the factorization method \cite{full-num-facto},
the proposed two indicator functions are able to give a surprising high resolution reconstruction of $\pa\Om$, which we think is a big improvement. 
Numerous numerical examples show that the proposed direct sampling methods work very well for general sources, even if the source support $\Om$ has some concave or non simply connected structures. 

The rest of the paper is organized as follows. In section 2, we prove a key lemma to characterize some special boundary points from the  multi-frequency far field patterns at two opposite observation directions.
Based on this lemma, we show in Section 3 the uniqueness results if the source support $\Om$ is composed of polygons and annuluses. 
Based on the uniqueness arguments, in section 4, we introduce two indicator functions for reconstructing the support of the source. Extended numerical examples are presented to show the validity and robustness of the proposed direct sampling methods.

\section{The characterization of special boundary points using the multi-frequency far field patterns at two opposite  directions}
\setcounter{equation}{0}
For clarity, we impose the following assumptions on the source function $f$ and its support $\Omega$:
\begin{assumption}\label{ass-source}
For the source function $f$ and its support $\Omega$, we assume that
\begin{itemize}
    \item $\Omega=\bigcup\limits_{m=1}^M \Omega_m$, where $\Omega_m$ are compact and $\Omega_m\cap\Omega_n=\emptyset
    ,\ \forall\ m\neq n $. The boundary $\partial\Omega_m$  is piecewise $C^1$ with finite many corners, and the interior of each $\Omega_m$ is a connected domain, $m=1,2,\cdots,M$. Note that each component $\Omega_m$ is not necessary simply connected.
    \item For any $\ \hx\in\mathbb S^1$, define a family of lines $l_{\hx,s}:=\left\{y\in\mathbb R^2\Big|\hx\cdot y=s \right\},\ s\in\mathbb R$. Then for any $ s\in\mathbb R$, there exist $\alpha_j(s)$, $\beta_j(s)$, and $J(s)<\infty$, satisfying
    \begin{align*}
        \beta_j(s)\leq\alpha_j(s),\quad&{\rm for}\ j=1,2,\cdots,J(s),\\
        \alpha_j(s)\leq\beta_{j+1}(s),\quad&{\rm for}\ j=1,2,\cdots,J(s)-1,
    \end{align*}
    and $l_{\hx,s}\cap\Omega=\bigcup\limits_{j=1}^{J(s)}\left\{s\hx+\tau\hx^{\perp}\in\mathbb R^2\Big|\beta_j(s)\leq\tau\leq\alpha_j(s)\right\}$.
    Here, $\hat{x}^{\perp}$ is obtained by rotating $\hat{x}$ anticlockwise by $\pi/2$, the illustration of the definitions of $\alpha_j$ and $\beta_j$ can be seen in Figure \ref{Ill-abj}.
    \item $f|_{\Om_m}\in C^1(\Omega_m)$, $m=1,2,\cdots,M$, and $f(y)\neq 0,\ \forall\ y\in\partial\Omega$.
\end{itemize}
\end{assumption}
\begin{figure}[h]
    \centering
    \includegraphics[width=0.3\linewidth]{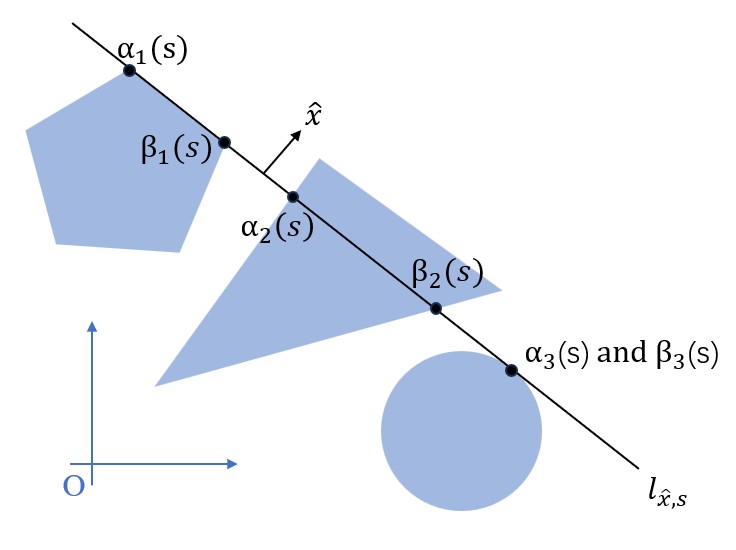}
    \caption{Illustration for $\alpha_j$ and $\beta_j$.}
    \label{Ill-abj}
\end{figure}

Note that the far field pattern $u^{\infty}(\cdot,k)$ depends analytically on the $k$, we have the far field patterns $u^{\infty}(\hx,k)$ for all $k>0$ from the measurement $u^{\infty}(\hat{x},k)$ with $k\in(k_{\min},k_{\max})$. Furthermore, we define 
\ben
u^\infty(\hx,0):=\lim_{k\rightarrow 0}u^\infty(\hx,k)\quad\mbox{and}\quad u^\infty(\hx,k):= u^{\infty}(-\hx,-k),\ \forall\ k<0,
\enn
and thus we can define the function $I_{\hat{x}}: \mathbb R\to\mathbb C$  from a knowledge of  multi-frequency far field patterns $\left\{u^{\infty}(\pm\hx,k)\Big|k\in\mathbb R\right\}$ ,
\begin{equation}   
\label{func_Is_C}
    \begin{aligned}
          I_{\hat{x}}(s):&=\int^{+\infty}_0\left[u^{\infty}(\hat{x},k)e^{iks}+ u^{\infty}(-\hat{x},k)e^{-iks}\right]dk\\
      &=\int^{+\infty}_{-\infty}u^{\infty}(\hat{x},k)e^{iks}dk\\
      &=\int^{+\infty}_{-\infty}e^{iks}\int^{+\infty}_{-\infty}\wi{f}(t)e^{-ikt}dtdk\\
      &=\wi{f}(s),\quad  s\in\mathbb R.  
    \end{aligned}
\end{equation}
Here, $\wi{f}(t):=\int^{+\infty}_{-\infty}f_{\hat{x}}(t,\tau)d\tau$ and $f_{\hat{x}}(t,\tau):=f(t\hx+\tau\hx^{\perp})$.
Note that if $f$ is a real valued function, then $\left\{u^{\infty}(\hx,k)\Big|k\in(k_{\min},k_{\max})\right\}$ is sufficient to obtain $I_{\hx}$, since $\overline{u^{\infty}(\hx,k)}=u^{\infty}(-\hx,k)$.

Denote by $X(\pa \Om)\subset \pa\Om$ the set of all corners on $\pa\Om$. Here and throughout the paper we say $y\in\pa \Om$ is a corner if the unit normal $\nu$ is not continuous at $y$. For a fixed direction $\hx\in \mathbb S^1$, we denote by $Y_{\hx}(\pa \Om)\subset \pa\Om$ the set of all the points on $\pa\Om$ whose normal direction is parallel to $\hx$. The following key lemma shows the potential of the function $I_{\hat{x}}$ given in \eqref{func_Is_C} to determine the points in $A_{\hx}:=X(\pa\Om)\cup Y_{\hx}(\pa \Om)$. 

Actually, $\Tilde{f}$ is the Radon transform $Rf(\hx,t)$ of source $f$, the properties  about how singularities in $f$ are connected to singularities in $Rf$  have been analyzed in \cite{Radon-qui,Radon-ram}. The results of Lemma \ref{Rela_betweenIandCorner} for more general sources have been studied in \cite{Radon-qui,Radon-ram}. We give a  simplified and elementary proof under the Assumption \ref{ass-source}, where some details will be applied in the proof of the subsequent uniqueness theorems.

\begin{lemma}\label{Rela_betweenIandCorner}
 Let Assumption \ref{ass-source} hold. For any fixed observation direction $\hx\in \mathbb S^1$, if the derivative of $I_{\hx}$ does not exist at $s_0\in \mathbb R$, then the line $l_{\hx,s_0}$  must pass through some point in $X(\pa\Om)\cup Y_{\hx}(\pa \Om)$.
\end{lemma}
\begin{proof}
    By the assumption of $\Om$ at the beginning of this subsection, we rewrite $I_{\hx}$ in the form
    \begin{align*}
        I_{\hx}(s)=\sum\limits_{j=1}^{J(s)}\int_{\beta_j(s)}^{\alpha_j(s)}f_{\hat{x}}(s,\tau)d\tau,\quad s\in\mathbb R.
    \end{align*}
    If $I_{\hx}^\prime(s)$ exist, straightforward calculations show that 
    \begin{align}\label{Iprime}
        I_{\hx}^{\prime}(s)=\sum\limits_{j=1}^{J(s)}\left[\int_{\beta_j(s)}^{\alpha_j(s)}\frac{\partial f_{\hat{x}}(t,\tau)}{\partial t}\Big|_{t=s}  d\tau+\alpha_j^{\prime}(s)f_{\hat{x}}(s,\alpha_j(s))- \beta_j^{\prime}(s)f_{\hat{x}}(s,\beta_j(s))\right], \quad s\in\mathbb R,
    \end{align}
where, for $j=1,2,\cdots,J(s)$,
  \begin{align}\label{ab-tan}
  \begin{split}
      \alpha_j^{\prime}(s)&=\tan\left(\langle\nu(s\hx+\alpha_j(s)\hx^{\perp}),\hx\rangle-\frac{\pi}{2}\right),\quad s\in\mathbb R,\\
      \beta_j^{\prime}(s)&=\tan\left(\langle\nu(s\hx+\beta_j(s)\hx^{\perp}),\hx\rangle-\frac{\pi}{2}\right),\quad s\in\mathbb R.
  \end{split}
  \end{align}
Here, $\nu$ denote the unit normal vector to the boundary $\partial \Omega$ directed into the exterior of $\Omega$ and
$\langle a, b\rangle\in(-\pi,\pi)$ is the angle rotated anticlockwisely from the unit vector $b$ to the unit vector $a$. Note that, for $j=1,2,\cdots,J(s)$,
\be \label{nu-tan}
\langle\nu(s\hx+\alpha_j(s)\hx^{\perp}),\hx\rangle \in(0,\pi)\quad\mbox{and}\quad
\langle\nu(s\hx+\beta_j(s)\hx^{\perp}),\hx\rangle \in(-\pi,0).
\en
From \eqref{Iprime}-\eqref{nu-tan} and assumptions on $f$ and $\Omega$, we deduce that the derivative $I^{\prime}_{\hx}$ has at most finitely many  discontinuous points. 

Denote by $R[f]$ and $L[f]$, respectively, the right limit and the left limit of the function $f$, we define $[\![f]\!]:=R[f]-L[f]$. 
If $I^{\prime}_{\hx}$ does not exist at $s_0$, we have $[\![I^{\prime}_{\hx}]\!](s_0)\neq 0$. By the representation \eqref{Iprime}, this happens only in one of the following five cases. 
\begin{itemize}
    \item \textbf{Case one:} $\left[\!\left[\int_{\beta_{j_0}(\cdot)}^{\alpha_{j_0}(\cdot)}\frac{\partial f_{\hat{x}}(t,\tau)}{\partial t}\Big|_{t=\cdot}  d\tau\right]\!\right](s_0)\neq0$ for some $1\leq j_0\leq J(s_0)$.
  \begin{figure}[htbp]
    \centering
    \includegraphics[width=.20\textwidth]{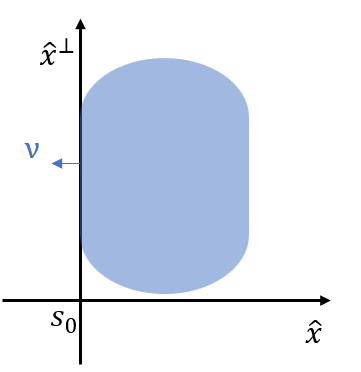}
    \caption{A simulation for the first case.}
    \label{situation-1}
\end{figure}  

In this case, considering the fact that $f|_{\Om_m}\in C^1(\Omega_m)$ and $\partial\Omega_m$ is piecewise $C^1$ with finite many corners, $m=1,2,\cdots,M$, we deduce that there exists  $a$ and $b$ such that
\begin{align*}
       \beta_{j_0}(s_0)\leq b<a\leq\alpha_{j_0}(s_0)\quad\mbox{and}\quad\left\{s_0\hx+\tau\hx^{\perp}\in\mathbb R^2\Big|b\leq\tau\leq a\right\}\subset l_{\hx,s_0}\cap\partial\Omega.
   \end{align*}
Figure \ref{situation-1} gives a simulation for this case. Obviously, 
\ben
\nu(y)\parallel\hx,\quad\forall y\in \left\{s_0\hx+\tau\hx^{\perp}\in\mathbb R^2\Big|b\leq\tau\leq a\right\},
\enn
which implies that the line $l_{\hx,s_0}$  passes through some points on $Y_{\hx}(\pa \Om)$.

    

\end{itemize}

Note that as long as case one does not occur, the integration part $\int_{\beta_{j}(\cdot)}^{\alpha_{j}(\cdot)}\frac{\partial f_{\hat{x}}(t,\tau)}{\partial t}\Big|_{t=\cdot}  d\tau$ in \eqref{Iprime} will not contributes discontinuity to $I'_{\hx}$. Thus, in the subsequent  four cases, we always assume that the case one does not occur, i.e. $\left[\!\left[\int_{\beta_{j}(\cdot)}^{\alpha_{j}(\cdot)}\frac{\partial f_{\hat{x}}(t,\tau)}{\partial t}\Big|_{t=\cdot}  d\tau\right]\!\right](s_0)=0,\ 1\leq j\leq J(s_0)$.  
\begin{itemize}
      \item \textbf{Case two:} $R[\alpha_{j_0}](s_0)=R[\beta_{j_0}](s_0)$(or $L[\alpha_{j_0}](s_0)=L[\beta_{j_0}](s_0)$), but $L[\alpha_{j_0}](s_0)$ and $L[\beta_{j_0}](s_0)$(or $R[\alpha_{j_0}](s_0)$ and $R[\beta_{j_0}](s_0)$) are undefined for some $1\leq j_0\leq J(s_0)$ .
   \begin{figure}[htbp]
    \centering
    \includegraphics[width=.20\textwidth]{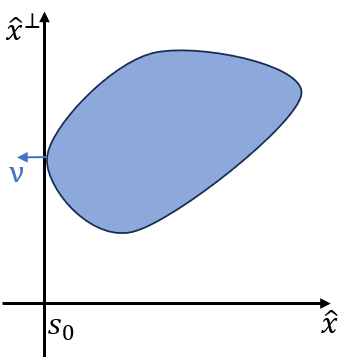}
    \caption{A simulation for the second case.}
    \label{situation-2}
\end{figure}

   As shown in Figure \ref{situation-2}, $R[\alpha_{j_0}](s_0)=R[\beta_{j_0}](s_0)$ while both $L[\alpha_{j_0}](s_0)$ and $L[\beta_{j_0}](s_0)$ are undefined. In this case, we have
   \begin{align}\label{discon-sec}
       [\![I_{\hx}^{\prime}]\!](s_0)=\left[R[\alpha_{j_0}^{\prime}](s_0)-R[\beta_{j_0}^{\prime}](s_0)\right]f(y_0).
   \end{align}
   where $y_0:=s_0\hx+R[\alpha_{j_0}](s_0)\hx^{\perp}\in\pa \Om$.
   By the assumption that $f\neq 0$ on $\pa \Om$, $[\![I^{\prime}_{\hx}]\!](s_0)\neq 0$ implies that $R[\alpha_{j_0}^{\prime}](s_0)-R[\beta_{j_0}^{\prime}](s_0)\neq0$. If both $R[\alpha_{j_0}^{\prime}](s_0)$ and $R[\beta_{j_0}^{\prime}](s_0)$ are infinity, with the help of \eqref{ab-tan}-\eqref{discon-sec}, we have $\nu(y_0)\parallel\hx$, that is the line $l_{\hx,s_0}$  passes through a point $y_0\in Y_{\hx}(\pa \Om)$. Otherwise, the point $y_0$ is a corner.
   
   Similarly, if $L[\alpha_{j_0}](s_0)=L[\beta_{j_0}](s_0)$ while both $R[\alpha_{j_0}](s_0)$ and $R[\beta_{j_0}](s_0)$ are undefined,  the line $l_{\hx,s_0}$  passes through a point $y_0\in X(\pa\Om)\cup Y_{\hx}(\pa \Om)$.

   \item \textbf{Case three:} $R[\alpha_{j_0}](s_0)=R[\beta_{j_0+1}](s_0)$(or $L[\alpha_{j_0}](s_0)=L[\beta_{j_0+1}](s_0)$), but $L[\alpha_{j_0}](s_0),\ L[\beta_{j_0+1}](s_0)$(or $R[\alpha_{j_0}](s_0),\ R[\beta_{j_0+1}](s_0)$) are undefined for some $1\leq j_0\leq J(s_0)-1$.
      \begin{figure}[htbp]
    \centering
    \includegraphics[width=.20\textwidth]{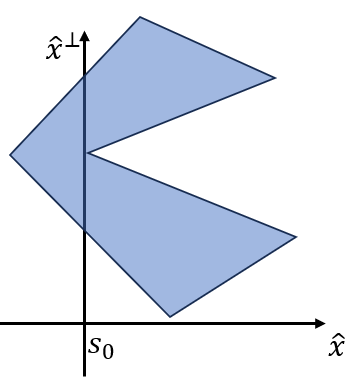}
    \caption{A simulation for the third case.}
    \label{situation-3}
\end{figure}

    As shown in Figure \ref{situation-3}, $R[\alpha_{j_0}](s_0)=R[\beta_{j_0+1}](s_0)$ while both $L[\alpha_{j_0}](s_0)$ and $L[\beta_{j_0+1}](s_0)$ are undefined. In this case,
     \begin{align}\label{discon-thr}
       [\![I_{\hx}^{\prime}]\!](s_0)=\left[R[\alpha_{j_0}^{\prime}](s_0)-R[\beta_{j_0+1}^{\prime}](s_0)\right]f(y_0).
   \end{align}
   where $y_0:=s_0\hx+R[\alpha_{j_0}](s_0)\hx^{\perp}\in\pa \Om$. Following the arguments for the second case, the line $l_{\hx,s_0}$  passes through a point $y_0\in X(\pa\Om)\cup Y_{\hx}(\pa \Om)$.
\end{itemize}
 The case two and the case three show that the discontinuities can be caused by the undefinition of $L[\alpha_j],R[\alpha_j],L[\beta_j],$ and $R[\beta_j]$. Note that the failure of the first three cases implies that $R[\alpha_j](s_0), L[\alpha_j](s_0), R[\beta_j](s_0)$ and $L[\beta_j](s_0)$ are all well defined and 
   \be\label{ABequal0}
   \left\{
   \begin{array}{lll}
       [\![\alpha_j]\!](s_0)=[\![\beta_j]\!](s_0)=0\ &{\rm for}\ j=1,2,\cdots,J(s_0),\\
        \alpha_j(s_0)>\beta_j(s_0)\ &{\rm for}\ j=1,2,\cdots,J(s_0),\\
        \alpha_j(s_0)<\beta_{j+1}(s_0)\ &{\rm for}\ j=1,2,\cdots,J(s_0)-1.
   \end{array}\right.
   \en
   Based on \eqref{ABequal0}, we have following two remaining cases.
\begin{itemize}
   \item \textbf{Case four:} $[\![\nu(\cdot\hx+\alpha_{j_0}(\cdot)\hx^{\perp})]\!](s_0)\neq 0$ (or $[\![\nu(\cdot\hx+\beta_{j_0}(\cdot)\hx^{\perp})]\!](s_0)\neq 0$) for some $1\leq j_0\leq J(s_0)$.
  \begin{figure}[htbp]
    \centering
    \includegraphics[width=.18\textwidth]{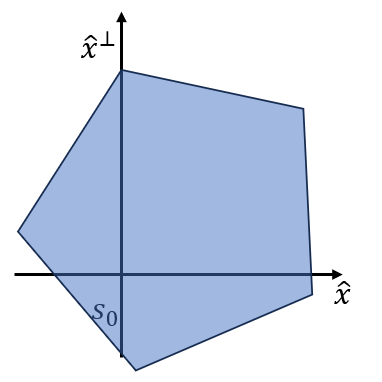}
    \caption{A simulation for the forth case.}
    \label{situation-4}
\end{figure}

   Figure \ref{situation-4} shows the case of $[\![\nu(\cdot\hx+\alpha_{j_0}(\cdot)\hx^{\perp})]\!](s_0)\neq 0$. In this case
   \begin{align}\label{discon-for}
       [\![I_{\hx}^{\prime}]\!](s_0)=[\![\alpha_{j_0}^{\prime}]\!](s_0)f(y_0).
   \end{align}
   with $y_0:=s_0\hx+\alpha_{j_0}(s_0)\hx^{\perp}$
   Obviously, $y_0\in X(\pa \Om)$ and $y_0\in l_{\hx, s_0}$.

   \item \textbf{Case five:} $[\![\nu(\cdot\hx+\alpha_{j_0}(\cdot)\hx^{\perp})]\!](s_0)= 0$ (or $[\![\nu(\cdot\hx+\beta_{j_0}(\cdot)\hx^{\perp})]\!](s_0)= 0$) , but $R[\alpha_{j_0}^{\prime}](s_0),L[\alpha_{j_0}^{\prime}](s_0)$ are infinite (or $R[\beta_{j_0}^{\prime}](s_0),L[\beta_{j_0}^{\prime}](s_0)$ are infinite) for some $1\leq j_0\leq J(s_0)$.
   \begin{figure}[htbp]
    \centering
    \includegraphics[width=.20\textwidth]{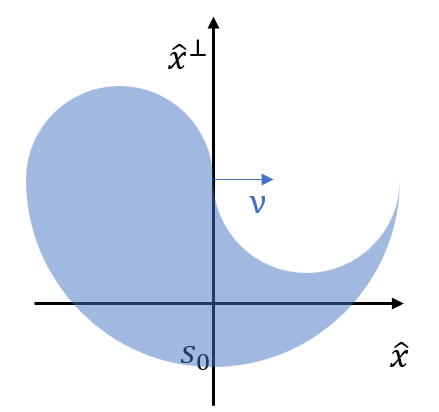}
    \caption{A simulation for the fifth case.}
    \label{situation-5}
\end{figure}

   Figure \ref{situation-5} shows a simulation for $L[\alpha_{j_0}^{\prime}](s_0)=R[\alpha_{j_0}^{\prime}](s_0)=\infty$. From \eqref{ab-tan}-\eqref{nu-tan} and \eqref{discon-for}, we deduce that 
   \begin{align*}
       \nu(s_0\hx+\alpha_{j_0}(s_0)\hx^{\perp})\parallel\hx\  {\rm and }\ s_0\hx+\alpha_{j_0}(s_0)\hx^{\perp}\in l_{\hx,s_0}\cap\partial\Omega.
   \end{align*}   

\end{itemize}
 The proof is complete. 
\end{proof}

We want to remark that the converse of the lemma \ref{Rela_betweenIandCorner} is not correct. The discontinuities contributed by the five situations in lemma \ref{Rela_betweenIandCorner} to the $I^{\prime}_{\hx}$ may cancel out each other, therefore some points in $X(\pa\Om)\cup Y_{\hx}(\pa \Om)$ may be missed by $I^{\prime}_{\hx}$ with a single observation direction. Of course, Lemma \ref{Rela_betweenIandCorner} implies that one may find all the corners with the increase of the number of the observation directions. Furthermore, under additional geometrical assumptions, we have the chance to uniquely determine $\Om$ from the multi-frequency far field patterns at finitely many observation directions. Such uniqueness results will be established in the next section.

\section{Uniqueness}
\setcounter{equation}{0}
This section is devoted to the uniqueness issue, we are interested in the inverse source problem of recovering $\Omega$ from the far field patterns $\left\{u^{\infty}(\pm\hx,k)\big|\hx\in\Theta_L,k\in(k_{\min},k_{\max})\right\}$.
To do so, we assume that $\Omega$ is composed of annuluses and polygons.
Using the fact that the far field pattern depends analytically on the wave number $k$, for each pair of observation directions $\pm\hx$, we are able to define the function $I_{\hx}$ as given by \eqref{func_Is_C}.

\subsection{Uniqueness for annular structured sources}
We begin with the annular structured sources, where $\Om$ is a combination of annuluses in the form
\begin{align}\label{f-ann}
    \Omega_m=\left\{y\in\mathbb R^2\big| r_m\leq | x-y_m| \leq R_m\right\} \  {\rm for} \   m=1,2,\cdots,M.
\end{align}
The inner diameter $r_m\in\mathbb R_{\geq 0}$, the outer diameter $ R_m\in\mathbb R^{+}$ and the center $y_m\in\mathbb R^2$ of the annuluses $\Omega_m$, $m=1,2\cdots,M$, are to be determined by the multi-frequency far field patterns at sparse observation directions.

Let $C(x,d):=\{y\in\mathbb R^2|\ |y-x|=d\}$ be a circle centered at $x$ with radius $d$, the following lemma is a corollary of Lemma \ref{Rela_betweenIandCorner}, which shows that some lines $l_{\hx,s}$ can be obtained from far field patterns $\left\{u^{\infty}(\pm\hx,k)\Big|k\in\mathbb R\right\}$.

\begin{lemma}\label{converse-ann}
    Let Assumption \ref{ass-source} hold. For any $s\in\mathbb R$, if there exists an unique $m_0\in\{1,2,\cdots,M\}$ such that the line $l_{\hx,s}$  are only tangent to the circle $C(y_{m_0},r_{m_0})$ or $C(y_{m_0},R_{m_0})$, then the line $l_{\hx,s}$ can be uniquely determined by far field patterns $\left\{u^{\infty}(\pm\hx,k)\Big|k\in\mathbb R\right\}$.
\end{lemma}
\begin{proof}
By the arguments of the second and third cases in Lemma \ref{Rela_betweenIandCorner}, the discontinuities of $I'_{\hx}$ at $s$ can't be canceled out because line $l_{\hx,s}$  are only tangent to the circle $C(y_{m_0},r_{m_0})$ or $C(y_{m_0},R_{m_0})$. Therefore, this line can be determined.
\end{proof}

\begin{theorem}\label{uni-ann}
 Under the assumption \ref{ass-source}, the annuluses $\Omega_m$, $m=1,2\cdots,M$, can be uniquely determined by far field patterns $\left\{u^{\infty}(\pm\hx,k)\big|\hx\in\Theta_L,k\in(k_{\min},k_{\max})\right\}$ with
    \begin{align}\label{Lin3.1}
        L> 8M-4.
    \end{align}
\end{theorem}
\begin{proof}
Our proof is then divided into two  steps.

{\bf Step one: \em Uniqueness of circles $C(y_m, d_m),\, d_m=r_m, R_m, m=1,2,\cdots, M$.}

We claim that, for any $C(y_{m_0},d_{m_0})\subset\partial\Omega_{m_0}$, there are at least $2L-8(M-1)$ lines which are tangent to the $C(y_{m_0},d_{m_0})$ can be determined by far field patterns $\left\{u^{\infty}(\pm\hx,k)\big|\hx\in\Theta_L,k\in(k_{\min},k_{\max})\right\}$. 
Actually, with a fixed $\hx\in\Theta_L$, there are two parallel lines $l_{\hx,s_j},\ j=1,2$ which are tangent to the $C(y_{m_0},d_{m_0})$. Since any two directions in $\Theta_L$ are not collinear,  we have a total of $2L$ lines that are tangent to $C(y_{m_0},d_{m_0})$. Among these lines, there are at most 
$8(M-1)$ common tangent lines that are also tangent to other circles $C(y_{m_1},r_{m_1})$(or $C(y_{m_1},R_{m_1})$) with $m_1\neq m_0$. According to the lemma \ref{converse-ann},  we deduce that the remaining $2L-8(M-1)$ tangent lines to $C(y_{m_0},d_{m_0})$ can be uniquely determined.

We turn to the case of a circle $C(y,d)\not\subset\partial\Omega$. If a line $l_{\hx,s}$, $\hx\in\Theta_L$ which is tangent to the $C(y,d)$ can be determined by far field patterns $\left\{u^{\infty}(\pm\hx,k)\big|\hx\in\Theta_L,k\in(k_{\min},k_{\max})\right\}$, then the line $l_{\hx,s}$ must be common tangent line that are tangent to $C(y_{m},r_{m})$(or $C(y_{m},R_{m})$) for some $1\leq m\leq M$. Thus, we can obtain at most $8M$ lines that are tangent to the $C(y,d)$.

Under the condition \eqref{Lin3.1}, we have  $2L-8(M-1)>8M$, that is, the circles $C(y_m, d_m),\, d_m=r_m, R_m$, produce more tangent lines than any other circles. Therefore, all the circles $C(y_m, d_m),\, d_m=r_m, R_m$ are uniquely determined by the far field patterns $\left\{u^{\infty}(\pm\hx,k)\big|\hx\in\Theta_L,k\in(k_{\min},k_{\max})\right\}$. 
    
{\bf Step two: \em Uniqueness of annuluses $\Om_m, \,m=1,2,\cdots,M$.}

Finally, we collect the concentric circles $C(y,d_j)$ we have found in the first step. Assume that
$d_1<d_2<\dots,d_J$. If $J$ is an even number, then $\Omega_{x,j}=\left\{y\in\mathbb R^2\big| d_{2j-1}\leq | x-y| \leq d_{2j}\right\}$, $j=1,2,\cdots, J/2$, are part of the annuluses we are looking for.
If $J$ is an odd number, then $\Omega_{x,1}=\left\{y\in\mathbb R^2\big| | x-y| \leq d_{1}\right\}$ and $\Omega_{x,j}=\left\{y\in\mathbb R^2\big| d_{2j}\leq | x-y| \leq d_{2j+1}\right\}$, $j=1,2,\cdots, (J-1)/2$, are part of the annuluses we are looking for. Following this procedure, since $M$ is finite, we obtain all the components of $\Om$.
\end{proof}

\subsection{Uniqueness for polygonal structured sources}
In this subsection, we assume that each $\Omega_m$ is a simply connected polygon, denote by $P_1,P_2,\cdots,P_N$ and $e_1,e_2,\cdots,e_N$, respectively, the corners and the edges of the $\partial\Omega$. the following lemma is again a corollary of Lemma \ref{Rela_betweenIandCorner}, which also shows that some lines $l_{\hx,s}$ can be obtained from far field patterns $\left\{u^{\infty}(\pm\hx,k)\Big|k\in\mathbb R\right\}$.

\begin{lemma}\label{converse-plo}
    Let Assumption \ref{ass-source} hold. For any $s\in\mathbb R$, if there exists an unique $n_0\in\{1,2,\cdots,N\}$ such that the line $l_{\hx,s}$ only pass through the corner $P_{n_0}$, then the line $l_{\hx,s}$ can be uniquely determined by far field patterns $\left\{u^{\infty}(\pm\hx,k)\Big|k\in\mathbb R\right\}$.
\end{lemma}
\begin{proof}
From the lemma \ref{Rela_betweenIandCorner}, if $I_{\hx}^{\prime}$ does not exist at $s$, then either the line $l_{\hx,s}$ pass through some corners or there are $e_j\subset l_{\hx,s}$ for some $1\leq j\leq N$. By the arguments of the second, the third, and the forth cases in Lemma \ref{Rela_betweenIandCorner}, if the line $l_{\hx,s}$ only pass through a single corner of $\partial\Omega$ (in this case, no edge lies on $l_{\hx,s}$), then the discontinuities of $I'_{\hx}$ at $s$ can't be canceled out. Therefore, this line can be determined.
\end{proof}

\begin{theorem}\label{uni-pol}
    Let Assumption \ref{ass-source} hold. If every $\Omega_m$ is a simply connected polygon, then all the corners $P_1,P_2,\cdots,P_N\in X(\pa\Om)$ can be uniquely determined by far field patterns $\left\{u^{\infty}(\pm\hx,k)\big|\hx\in\Theta_L,k\in(k_{\min},k_{\max})\right\}$ with
    \begin{align}\label{Lin3.2}
        L> 2N-1.
    \end{align}
\end{theorem}
\begin{proof}
With a fixed observation direction $\hx\in\Theta_L$, from the lemma \ref{converse-plo}, if the line $l_{\hx,s}$ only pass through a single corner of $\partial\Omega$, then the line $l_{\hx,s}$ can be obtained.


For a corner $P_0\in X(\pa\Om)$, we are able to find at least $L-(N-1)$ lines which pass through $P_0$ by the indicator function $I_{\hx}, \hx\in\Theta_L$, which is defined by the far field patterns $\left\{u^{\infty}(\pm\hx,k)\big|\hx\in\Theta_L,k\in(k_{\min},k_{\max})\right\}$.
Actually, since any two directions in $\Theta_L$ are not collinear, we have $L$ lines $l_{\hx,\hx\cdot P_0}$ passing through $P_0$. Among these lines, there are at most $N-1$ lines that can pass through $P_0$ and another corner simultaneously. Therefore, we finally obtain at least $L-(N-1)$ lines passing through $P_0$.

We now turn to the points $P\in\mathbb R^2\backslash\{P_1,P_2,\cdots,P_N\}$. If the line $l_{\hx,\hx\cdot P},\ \hx\in\Theta_L$ can be determined by the indicator function $I_{\hx}$, we deduce that $\hx\cdot P_j=\hx\cdot P$ for some $1\leq j\leq N$. Consequently, there are at most $N$ lines can be determined.

Under the condition \eqref{Lin3.2}, comparing the number of lines determined by the indicator functions $I_{\hx}, \hx\in\Theta_L$, we deduce that all the corners can be determined by the far field patterns $\left\{u^{\infty}(\pm\hx,k)\big|\hx\in\Theta_L,k\in(k_{\min},k_{\max})\right\}$. 
\end{proof}

Note that from the representation \eqref{Iprime}-\eqref{discon-thr} and \eqref{discon-for}, $f|_{\{P_1,P_2,\cdots,P_N\}}\neq 0$ is sufficient to establish the theorem \ref{uni-pol}, we don't need to further request that $f(y)\neq 0, \ \forall y\in\partial\Omega$.
\begin{theorem}\label{uni-pol-edge}
   Let Assumption \ref{ass-source} hold. If every $\Omega_m$ is a simply connected polygon and all the corners are known, we can choose  finite observation direction $\hx_i,\ 1\leq i\leq T$, such that all the edges $e_1,e_2,\cdots,e_N$ as well as $f(P_1),f(P_2),\cdots,f(P_N)$ can be determined by the far field patterns $\left\{u^{\infty}(\pm\hx_i,k)\big|1\leq i\leq T,k\in(k_{\min},k_{\max})\right\}$.
\end{theorem}
\begin{proof}
Taking $\hx_0\in\mathbb S^1$ such that $\hx_0\cdot P_i\neq\hx_0\cdot P_j,\ \forall\ i\neq j$. After relabeling we can assume that
 \begin{align}\label{x0-chosen}
     \hx_0\cdot P_i<\hx_0\cdot P_{i+1},\ 1\leq i\leq N-1.
 \end{align}
 Then $I^{\prime}_{\hx_0}$ does not exist at $\hx_0\cdot P_i,\ 1\leq i\leq N$ and $[\![I^{\prime}_{\hx_0}]\!](\hx_0\cdot P_i)$ are known for $\ 1\leq i\leq N$. From the arguments of the second, the third, and the forth cases in Lemma \ref{Rela_betweenIandCorner} and the representation \eqref{Iprime}-\eqref{discon-thr} and \eqref{discon-for}, we deduce that $[\![I^{\prime}_{\hx_0}]\!](\hx_0\cdot P_i),\ 1\leq i\leq N$ are finite. Furthermore, for all possible edges passing through the point $P_i$ and all possible corresponding normal directions, we can calculate all the possible $f(P_i)$ and collect them in a set $F_i$.
 
For any corner $P_{n_0}$, we aim to identify whether $\overline{P_1P_{n_0}}$ is really an edge of $\Omega$ in two steps.
\begin{itemize}
    \item First, $P_{n_0}$ should be the closest corner to $P_1$ in the $\overrightarrow{P_1P_{n_0}}$ direction, otherwise it contradicts the assumption that all the simply connected $\Omega_m$ are disjoint with each other.
    \item Second, define $\mathcal{NP}(n_0):=\{ 1<n\leq N|\overrightarrow{P_1P_n}\  {\rm is \ not\ parallel \ to} \ \overrightarrow{P_1P_{n_0}}\}$ and
    \begin{align}\label{def-theta1}
        \theta_1:=\min\limits_{\mathcal{NP}(n_0)}\left\{
        \min\big\{\big|\langle\overrightarrow{P_1P_n},-\overrightarrow{P_1P_{n_0}}\rangle\big|,\big|\langle\overrightarrow{P_1P_n},\overrightarrow{P_1P_{n_0}}\rangle\big|\big\}\right\}\in(0,\frac{\pi}{2}].
    \end{align}
    Then, we can choose an observation direction $\hx_1$ such that $\hx_1\cdot P_1\neq\hx_1\cdot P_n,\ \forall \ 1<n\leq N$ and 
    \begin{align}\label{def-hx1}
        \big|\langle\hx_1,\overrightarrow{P_1P_{n_0}}^{\perp}\rangle\big|<\frac{\theta_1}{2}\quad {\rm and}\quad |\tan(\langle\hx_1,\overrightarrow{P_1P_{n_0}}\rangle)|>\left|\tan(\frac{\pi-\theta_1}{2})\right|\left(1+2\frac{\max\limits_{x\in F_1}|x|}{\min\limits_{x\in F_1}|x|}\right)
    \end{align}
    If $\overline{P_1P_{n_0}}$ is really an edge of $\Omega$ (suppose that another real edge is $\overline{P_1P_{n^{\star}}}$), then \eqref{Iprime}-\eqref{discon-thr} and \eqref{discon-for} implies that 
    \be\label{real edge}
        |[\![I^{\prime}_{\hx_1}]\!](\hx_1\cdot P_1)|&\geq& \left(|\tan(\langle\hx_1,\overrightarrow{P_1P_{n_0}}\rangle)|-|\tan(\langle\hx_1,\overrightarrow{P_1P_{n^\star}}\rangle)|\right)\min\limits_{x\in F_1}|x|\cr
        &\geq& \left(|\tan(\langle\hx_1,\overrightarrow{P_1P_{n_0}}\rangle)|-|\tan(\frac{\pi}{2}-\theta_1+\frac{\theta_1}{2})|\right)\min\limits_{x\in F_1}|x|.
    \en
If $\overline{P_1P_{n_0}}$ is not an edge of $\Omega$ (suppose that two real edge are $\overline{P_1P_{n^{\star}}}$ and $\overline{P_1P_{n_{\star}}}$), then we have
    \be\label{fake edge}
        |[\![I^{\prime}_{\hx_1}]\!](\hx_1\cdot P_1)|&\leq &\left(|\tan(\langle\hx_1,\overrightarrow{P_1P_{n_{\star}}}\rangle)|+|\tan(\langle\hx_1,\overrightarrow{P_1P_{n^\star}}\rangle)|\right)\max\limits_{x\in F_1}|x|\cr
        &\leq&2|\tan(\langle\frac{\pi}{2}-\theta_1+\frac{\theta_1}{2}\rangle)|\max\limits_{x\in F_1}|x|.
    \en
We can identify whether $\overline{P_1P_{n_0}}$ is really an edge of $\Omega$ with the help of \eqref{def-hx1}-\eqref{fake edge}, so we can identify two real edges passing through $P_1$ by repeating this procedure.
\end{itemize}
Similarly, we can gradually identify all real edges passing through $P_2,P_3,\cdots,P_N$, that is, all the edges $e_1,e_2,\cdots,e_N$ can be determined by far field patterns. Finally, with the help of \eqref{discon-sec}, \eqref{discon-thr}, and \eqref{discon-for}, the values $f(P_1),f(P_2),\cdots,f(P_N)$ can be computed from $[\![I^{\prime}_{\hx_0}]\!](\hx_0\cdot P_i),\ 1\leq i\leq N$.
The proof is complete.
\end{proof}

We have an elementary estimate of $T$.  In the worst case, besides $\hx_0$, we need at most $\binom{N}{2}$ observation directions to identify all possible edges.
So we have 
\begin{align*}
    T\leq1+\binom{N}{2}=\frac{N^2-N+2}{2}.
\end{align*}

\subsection{Uniqueness for mixed structured sources}
Finally, we consider a mixed structured source. Each connected component $\Omega_m$ is either a simple connected polygon or an annulus. Assume that $\Omega$ has $M$ annuluses and $N$ corners.  

Before introducing the uniqueness theorem, let's analyze the contribution of each $\Omega_m$ to the discontinuity of function $I^{\prime}_{\hx}$.

If $\Omega=A(y_0,r,R)$, then for any observation direction $\hx$, $I_{\hx}^{\prime}$ is only discontinuous at $\hx\cdot y_0\pm R$ and $ \hx\cdot y_0\pm r$. In addition, according to \eqref{discon-sec} and \eqref{discon-thr}, for $d=r,R$ we deduce that 
\be
\label{discon-ann-p}
[\![I_{\hx}^{\prime}]\!](\hx\cdot y_0+ d)&=& -2 G_df^d_{+}\lim\limits_{s\to(\hx\cdot y_0+ d)^{-}}\frac{s-\hx\cdot y_0}{\sqrt{d^2-(s-\hx\cdot y_0)^2}}\cr
    &=&-2G_df^d_{+}\lim\limits_{\varepsilon^+_d\to0^{+}}\frac{d-\varepsilon^+_d}{\sqrt{d^2-(d-\varepsilon^+_d)^2}}\cr
    &=&-G_d f^d_{+}\sqrt{2d}\lim\limits_{\varepsilon^+_d\to0^{+}}\frac{1}{\sqrt{\varepsilon^+_d}},
\en
and
\be\label{discon-ann-m}
[\![I_{\hx}^{\prime}]\!](\hx\cdot y_0-d)&=& -2G_df^d_{-}\lim\limits_{s\to(\hx\cdot y_0- d)^{+}}\frac{s-\hx\cdot y_0}{\sqrt{d^2-(s-\hx\cdot y_0)^2}}\cr
    &=&-2G_df^d_{-}\lim\limits_{\varepsilon^-_d\to0^{+}}\frac{-(d-\varepsilon_d^-)}{\sqrt{d^2-(d-\varepsilon_d^-)^2}}\cr
    &=&G_d f^d_{-}\sqrt{2d}\lim\limits_{\varepsilon_d^-\to0^{+}}\frac{1}{\sqrt{\varepsilon_d^-}}.
\en
Here, $f^{d}_{\pm}:=f|_{\partial\Omega\cap l_{\hx,(\hx\cdot y_0\pm d)}}$, denote the value of source $f$ at the tangent point of the line $l_{\hx,s}$ and the circle $C(y_0,d)$, meanwhile, $\varepsilon^{\pm}_d:=d\pm(\hx\cdot y_0-s)$ and $G_R=1,G_r=-1$.

If $\Omega$ is a polygon, we further suppose that observation direction $\hx$ satisfies that $e_0\subset l_{\hx,s}$ for some $s$ and an unique edge $e_0$. Let $e_0=\overline{P_{\star}P^{\star}}$ and denote by $\nu^{\star},\nu_{\star}$, respectively, the unit normal vector at $P^{\star}, P_{\star}$ which is not perpendicular to $\hx$. With the help of \eqref{Iprime} and \eqref{ab-tan}, we have 
\begin{align}\label{discon-pol}
   \left|[\![I_{\hx}^{\prime}]\!](s)  \right|\leq |\overline{P_{\star}P^{\star}}| \max\limits_{y\in e_0}|[\![\nabla f(y)\cdot\hx]\!]|+|\cot(\langle\nu_{\star},\hx\rangle)f(P_{\star})|+|\cot(\langle\nu^{\star},\hx\rangle)f(P^{\star})|
\end{align}
The representation \eqref{discon-ann-p}, \eqref{discon-ann-m}, \eqref{fake edge}, \eqref{discon-pol}, and linearity implies that the contribution of every annulus to the discontinuity of function $I^{\prime}_{\hx}$ is infinite ($O (1/\sqrt{\varepsilon_d^{\pm}})$), while the contribution of each polygon to the discontinuity of function $I^{\prime}_{\hx}$ is always finite.
\begin{theorem}\label{uni-mix}
 Let Assumption \ref{ass-source} hold. Assume further that $\Omega$ is composed of $M$ annuluses and some polygons with $N$ corners, then all the annuluses and the corners $P_1,P_2,\cdots,P_N$ can be uniquely determined by far field patterns $\left\{u^{\infty}(\pm\hx,k)\big|\hx\in\Theta_L,k\in(k_{\min},k_{\max})\right\}$ with
    \begin{align}\label{Lin3.4}
        L> \max\{8M-4, 4M+2N-1\}.
    \end{align}
\end{theorem}
\begin{proof}
   Note that the contribution of a corner $P_0$ to the discontinuity of function $I^{\prime}_{\hx}$ can only be canceled out by other corners. So do the annulus.

    Fixed an observation direction $\hx\in\Theta_L$, some parallel lines $l_{\hx,s}$ can be obtained from a knowledge of $I^{\prime}_{\hx}$. Furthermore, for those $s$ such that $[\![I^{\prime}_{\hx}]\!](s)=\infty$, line $l_{\hx,s}$ is tangent to the $C(y_m,r_m)$(or $C(y_m,R_m)$) for some $m=1,2,\cdots,M$. Thus, the same as the theorem \ref{uni-ann}, all the annuluses can be uniquely determined by far field patterns $\left\{u^{\infty}(\pm\hx,k)\big|\hx\in\Theta_L,k\in(k_{\min},k_{\max})\right\}$ with $ L\geq 8M-4.$
    
    For fixed corner $P_0$, similar to theorem \ref{uni-pol}, there are at least $L-(N-1)$ lines which pass through $P_0$ can be determined by far field patterns $\left\{u^{\infty}(\pm\hx,k)\big|\hx\in\Theta_L,k\in(k_{\min},k_{\max})\right\}$. 
    
    For any $P\in\mathbb R^2\backslash\{P_1,P_2,\cdots,P_N\}$, if the line $l_{\hx,\hx\cdot P},\ \hx\in\Theta_L$ can be determined by far field patterns $\left\{u^{\infty}(\pm\hx,k)\big|\hx\in\Theta_L,k\in(k_{\min},k_{\max})\right\}$, we deduce that $\hx\cdot P_j=\hx\cdot P$ for some $1\leq j\leq N$ or line $l_{\hx,\hx\cdot P}$ is tangent to the $C(y_m,r_m)$(or $C(y_m,R_m)$) for some $m=1,2,\cdots,M$. Thus, there are at most $4M+N$ lines can be determined.

    Then, all the corners can be determined by far field patterns with the condition \eqref{Lin3.4}.
\end{proof}

Obviously, we can choose finite observation direction $\hx_i,\ 1\leq i\leq T$, such that all the edges $e_1,e_2,\cdots,e_N$ as well as $f(P_1),f(P_2),\cdots,f(P_N)$ can be determined by far field patterns $\left\{u^{\infty}(\pm\hx_i,k)\big|1\leq i\leq T,k\in(k_{\min},k_{\max})\right\}$. This can be proved by following the steps in the theorem \ref{uni-pol-edge} and additionally require that $l_{\hx_j,\hx_j\cdot P_i},\ 1\leq j\leq T,\ 1\leq i\leq N$ is not tangent to the $C(y_m,r_m)$ or $C(y_m,R_m)$ for some $m=1,2,\cdots,M$. As for the annuluses,  with the help of \eqref{discon-ann-p}-\eqref{discon-ann-m}, we have actually obtained the values of the source function $f$ at the boundary points with the observation directions as the normal directions.
Finally, if it is further assumed the source $f$ is piecewise constant, the source $f$ can be uniquely determined by the multi-frequency far field patterns at sparse observation directions.

\section{Novel indicator functions and numerical simulations}
Recall that a generic sampling type method's scheme is as follows:
\begin{itemize}{\em
    \item Collect the measurements;
    \item Select a sampling region in $\R^2$ with a fine mesh $\mathcal{G}$ containing the unknown object $\Om$;
    \item Compute the indicator function $I(z)$ for all sampling point $z\in \mathcal{G}$;
    \item Plot the indicator function $I(z)$.}
\end{itemize}
The design of an indicator function is the key of a valuable sampling method.
In the subsection \ref{twoindicators}, we introduce two novel indicator functions to reconstruct the support $\Om$ of the source $f$ with the multi-frequency far field patterns $\left\{u^{\infty}(\hx,k)\big|\hx\in\Theta_L,k\in(k_{\min},k_{\max})\right\}$. The numerical simulations will be presented in the subsequent subsections 4.2-4.4. To begin with, we introduce the data to be used in the numerical simulations.
The sparse observation direction set is chosen to avoid some special directions as follow:
\ben
\Theta_L:=\left\{\Big(\cos\frac{(2l-0.7L) \pi}{L}, \sin\frac{(2l-0.7L) \pi}{L}\Big)\, \Big{|}\, l=0, 1, \cdots, L-1\right\}.
\enn
We take the wave numbers 
\ben
k_m=\frac{m}{2},\ m\in\mathbb Z,\  1\leq m\leq 2\Lambda.
\enn
Given a source function $f$, we compute
a synthetic approximation $u^{\infty}(\hx_l,k_m)$ from \eqref{u_inf_1} by the trapezoidal rule. We
further perturb this data by random noise in the form
 \begin{align*}
     u^{\infty,\delta}(\hx_l,k_m)=u^{\infty}(\hx_l,k_m)\Big(1+\delta\big[X_{l,m}+Y_{l,m}i\big]\Big),
 \end{align*}
where $X_{l,m},Y_{l,m}\sim N(0,1)$ are  independent random variables, $N(0,1)$ is a normal distribution with mean zero and standard derivation one. In all the experiments, we set $\delta=0.3$ and use a grid $\mathcal{G}$ of $P \times Q = 601\times 601$ equally spaced sampling points on the rectangle $[-3,3]\times[-3,3]$.



\subsection{Two indicator functions}\label{twoindicators}
\begin{itemize}
    \item \textbf{Indicator function $\mathcal{I}^-$.}
\end{itemize}
Motivated by the uniqueness analyses in the previous sections, we introduce the first indicator function $\mathcal{I}^-$ by
\begin{align}\label{I-z}
\mathcal{I}^-(z):=\frac{1}{L}\sum\limits_{\hx\in\Theta_L}|\mathcal{I}_{\hx}^{-}(z)|,\ z\in\mathcal{G},
\end{align}
where
\begin{align}\label{I_hx_z}
 \mathcal{I}^{-}_{\hat{x}}(z)=\int^{k_{2\Lambda}}_{k_{1}}k\left[u^{\infty,\delta}(\hat{x},k)e^{ik\hx\cdot z}-u^{\infty,\delta}(-\hat{x},k)e^{-ik\hx\cdot z}\right]dk,\quad z\in\mathcal{G}.
\end{align}
The function $\mathcal{I}^-_{\hx}(z)$ is an approximation of $I^{\prime}_{\hx}(\hx\cdot z)$. 

Alternatively, to quantify the discontinuity of the function $I^{\prime}_{\hx}(\hx\cdot z)$, we recall the Sobel matrices
\begin{equation*}
    S_1:=
    \begin{bmatrix}
        -1&0&1\\
        -2&0&2\\
        -1&0&1
    \end{bmatrix}\, {\rm and}\, 
    S_2:=
    \begin{bmatrix}
        -1&-2&-1\\
        0&0&0\\
        1&2&1
    \end{bmatrix}
\end{equation*}
and apply a technique which is widely used in image processing \cite{Sobel}. Precisely, we define the convolution of sampling matrix $\mathcal{I}^-_{\hx}\in\mathbb C^{P\times Q}$ induced by \eqref{I_hx_z} and the Sobel matrices $S_t\in\mathbb R^{3\times 3}, t=1,2$, by
\begin{equation}
    \mathcal{I}^-_{\hx}\otimes S_t(p,q):=\left\{
    \begin{aligned}
        &0,   \qquad{\rm if}\ p\in\{1,P\}\ {\rm or}\  q\in\{1,Q\},\\
    \sum\limits_{k,l=1}^3 &\mathcal{I}^-_{\hx}(p-k+2,q-l+2)S_t(k,l),  \quad {\rm otherwise}.
    \end{aligned}
    \right.
\end{equation}
Finally, we define the processed indicator function $\wi{\mathcal{I}}^-$ by
\begin{align}\label{IPs}
\wi{\mathcal{I}}^{-}(z):=\frac{1}{L}\sum\limits_{\hx\in\Theta_L}\wi{\mathcal{I}}_{\hx}^{-}(z),\ z\in\mathcal{G},
\end{align}
where
\begin{align}\label{Ip-hx}
    \wi{\mathcal{I}}_{\hx}^{-}(z)=\wi{\mathcal{I}}_{\hx}^{-}(p,q)=\frac{\sqrt{|\mathcal{I}_{\hx}^{-}\otimes S_1(p,q)|^2+|\mathcal{I}_{\hx}^{-}\otimes S_2(p,q)|^2}}{\max\limits_{r,t}\sqrt{|\mathcal{I}_{\hx}^{-}\otimes S_1(r,t)|^2+|\mathcal{I}_{\hx}^{-}\otimes S_2(r,t)|^2}}.
\end{align}
As mentioned in the previous sections, for fixed observation direction $\hx$, the value of the indicator function $\mathcal{I}^-_{\hx}$ is expect to blow up (or rapidly change) at points $z$ where $I_{\hx}^{\prime}(\hx\cdot z)$ does not exist. 
The indicator function $\mathcal{I}^{-}_{\hx}$ should be good at catching the tangent lines $l_{\hx,s}$ of the annuluses.
However, the processed indicator function $\wi{\mathcal{I}}^{-}_{\hx}$ is able to catch both the tangent lines $l_{\hx,s}$ of the annuluses and the lines $l_{\hx,s}$ passing through corners since the indicator function $\wi{\mathcal{I}}^{-}_{\hx}$ blows up on those lines.  

\begin{itemize}
    \item \textbf{Indicator function $\mathcal{I}^+$.}
\end{itemize}
We turn our attention to the second indicator function $\mathcal{I}^+$ defined by
\begin{align}\label{I+z}
    \mathcal{I}^+(z):=\frac{1}{4\pi L}\sum\limits_{\hx\in\Theta_L}\mathcal{I}^+_{\hx}(z),\ z\in\mathcal{G}.
\end{align}
where
\begin{align}\label{I_S_hx_z}
 \mathcal{I}^+_{\hat{x}}(z):=\int^{k_{2\Lambda}}_{k_{1}}k\left[u^{\infty,\delta}(\hat{x},k)e^{ik\hx\cdot z}+u^{\infty,\delta}(-\hat{x},k)e^{-ik\hx\cdot z}\right]dk,\quad z\in\mathcal{G}.
\end{align}
Note that the difference between \eqref{I_S_hx_z} and \eqref{I_hx_z} is just the sign of two integrands.
We would also like to remind that, different to the indicator function $\mathcal{I}^-(z)$ defined in \eqref{I-z}, we do not take the modulus before the summation with respect to the observation directions $\hx\in\Theta_L$. With these modifications, we find that the indicator function $\mathcal{I}^+(z)$ given in \eqref{I+z} can be used to approach the source function $f$ if there are sufficiently many observation directions. Actually, if $L$ and $k_{2\Lambda}$ are very large and $k_1$ is very small, with the help of \eqref{u_inf_1}, we have
\begin{align*}
    \mathcal{I}^+(z)&=\frac{1}{4\pi L}\sum\limits_{\hx\in\Theta_L}\int^{k_{2\Lambda}}_{k_{1}}\left[ku^{\infty,\delta}(\hat{x},k)e^{ik\hx\cdot z}+ku^{\infty,\delta}(-\hat{x},k)e^{-ik\hx\cdot z}\right]dk\\
    &\approx\frac{1}{8\pi^2}\int_{\mathbb S^1}\int^{+\infty}_{0}\left[ku^{\infty}(\hat{x},k)e^{ik\hx\cdot z}+ku^{\infty}(-\hat{x},k)e^{-ik\hx\cdot z}\right]dkds_{\hx}\\
    &=\frac{1}{4\pi^2}\int_{\mathbb R^2}\int_{\mathbb R^2}f(y)e^{-i\xi\cdot y}dy\ e^{i\xi\cdot z}d\xi\\
     &=f(z),\quad \forall\ z\in\mathbb R^2\backslash\partial\Omega.
\end{align*}
In the last equality, we exclude the sampling points on $\pa\Om$ because we have assumed that the source function $f$ has a jump across $\pa\Om$. Despite the above analysis requires a lot of observation directions, the subsequent numerical examples show that, even if we have only finitely many observation directions, the indicator function $\mathcal{I}^+(z)$ performs very well for reconstructing the source support $\Om$ and approaching the source function $f$.


\subsection{Three simple examples}
In this subsection, we consider three simple examples to verify the effectiveness and robustness of our indicator functions proposed in the previous subsection. We set $f=\chi_{\Om}$, where $\chi_{\Om}$ is the characteristic function and $\Om$ is defined by one of the following domains:
\begin{itemize}
    \item an annulus as shown  on the top left of Figure \ref{annulus-L_L_1};
    \item a L-shaped polygon as shown  on the bottom left of Figure \ref{annulus-L_L_1};
    \item a kite shaped domain $(x+\frac{13y^2}{15})^2+y^2\leq 1$.
\end{itemize}

Figure \ref{annulus-L_L_1} shows the reconstructions of the first two examples with $L=1$. The observation direction is $\hx = (\cos 0.7\pi, -\sin 0.7\pi)$. Clearly, both the indicator functions $\mathcal{I}^-$ and $\wi{\mathcal{I}}^-$ capture the tangent lines of the inner and outer circles of the annulus. Moreover, the processed indicator function  $\wi{\mathcal{I}}^-$ also accurately capture the lines passing through the corners. These lines are perpendicular to the observation direction.
   \begin{figure}[htbp]
    \centering
    \includegraphics[width=1\textwidth]{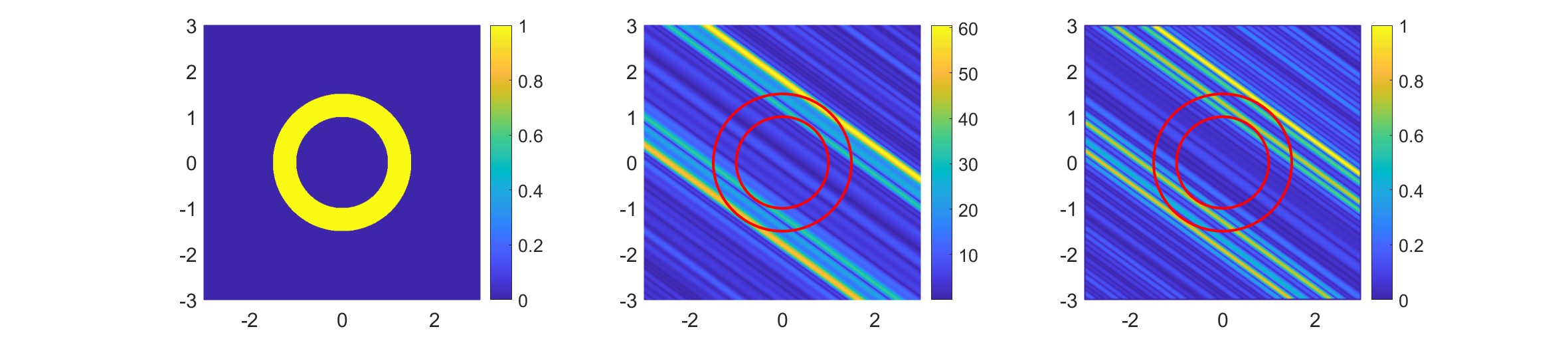}\\
    \includegraphics[width=1\textwidth]{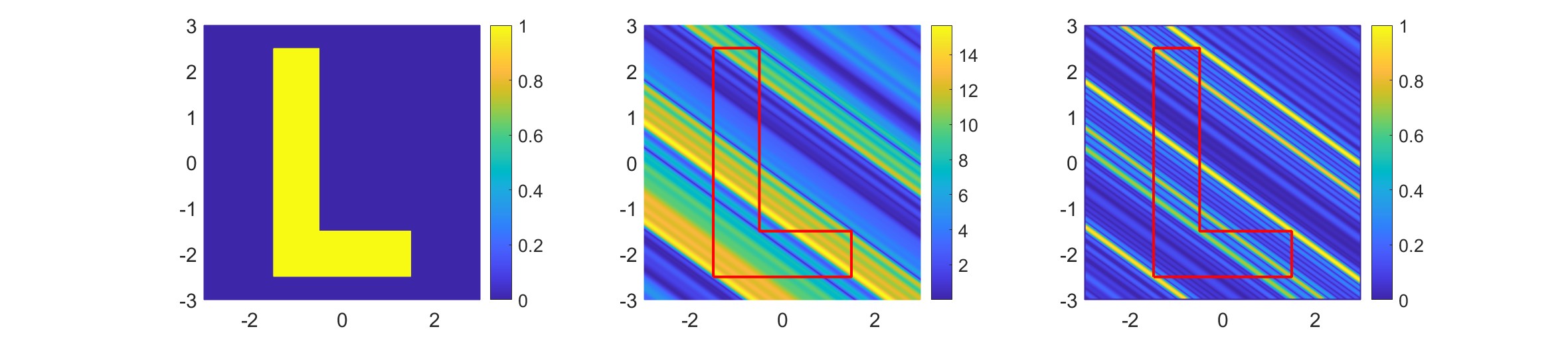}
    \caption{\em Reconstructions of the source support with $L=1,\Lambda=30$. Left: The true source functions. Middle: reconstructions by plotting $\mathcal{I}^-$, here $\partial\Omega$ is plotted by red curves. Right: reconstructions by plotting $\wi{\mathcal{I}}^{-}$, here  $\partial\Omega$ is plotted by red lines.}
    \label{annulus-L_L_1}
\end{figure}

Furthermore, Figure \ref{annulus-L_L_25} shows the reconstructions with $L=25$. Both the inner and outer circles of the annulus $\Omega$ are well reconstructed. As mentioned in \eqref{discon-ann-p}-\eqref{discon-ann-m}, the indicator function $\mathcal{I}^{-}$ takes larger value on the outer circle while smaller value on the inner circle.
All the edges of the polygon $\Omega$ have been reconstructed by indicator function $\mathcal{I}^-$. This phenomenon is reasonable since Theorem \ref{uni-pol-edge} implies that the value of $|\mathcal{I}^-_{\hx}|$ should be large on an edge when the observation direction $\hx$ is nearly perpendicular to this edge. Moreover,  the corners are reliably reconstructed by $\wi{\mathcal{I}}^{-}$, which is designed to emphasize line $l_{\hx,s}$ passing through the corners at the cost of ignoring edges' information. 
Based on these observations, to reconstruct the whole boundary $\pa\Om$, we will not show the reconstruction results of $\wi{\mathcal{I}}^-$ in the subsequent examples.
\begin{figure}[htbp]
    \centering
    \includegraphics[width=1\textwidth]{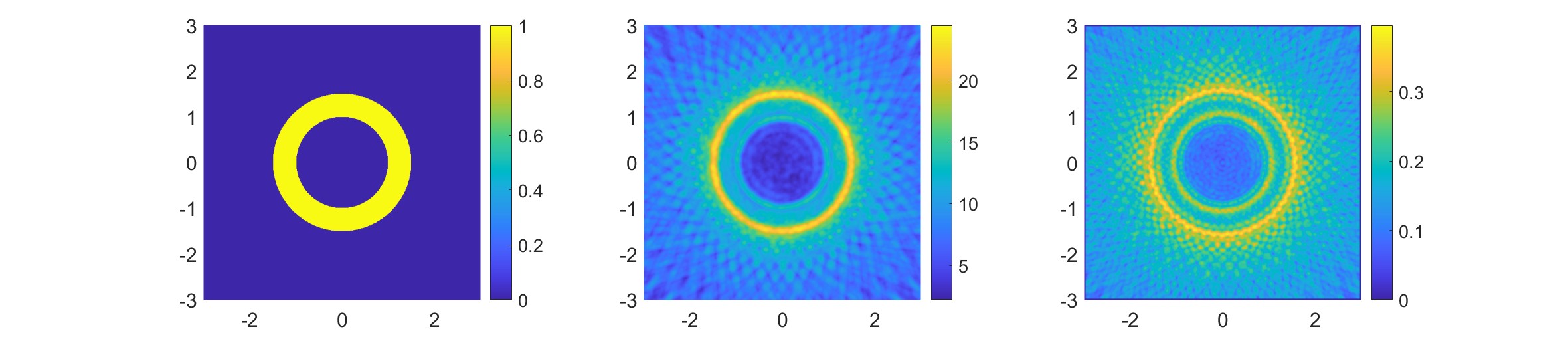}\\
    \includegraphics[width=1\textwidth]{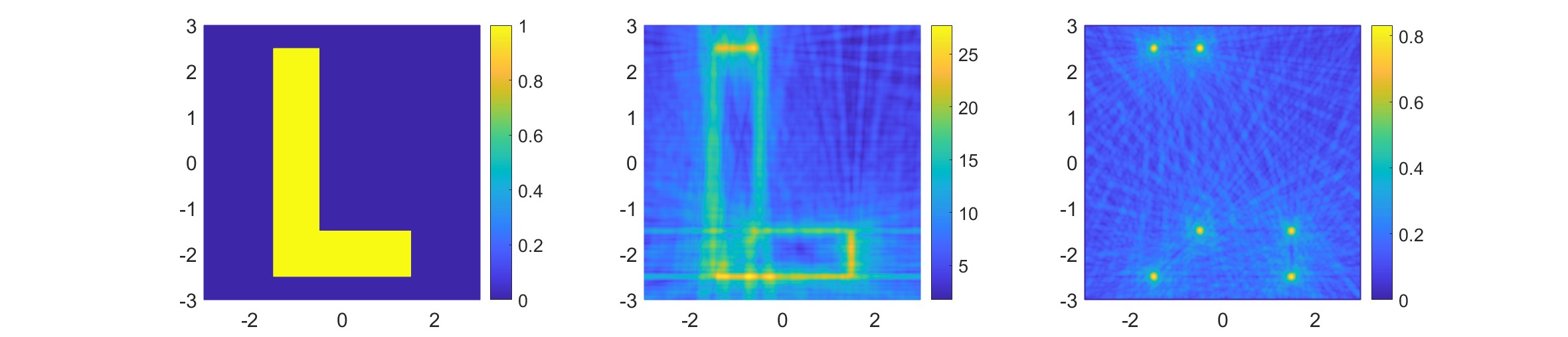}
    \caption{\em Reconstructions of the source support with $L=25,\Lambda=30$. Left: The true source functions. Middle: reconstructions by plotting $\mathcal{I}^-$. Right: reconstructions by plotting $\wi{\mathcal{I}}^{-}$.}
    \label{annulus-L_L_25}
\end{figure}

We take the kite shaped source as a benchmark example to compare the behaviours of the two indicator functions $\mathcal{I}^{\pm}$. Figure \ref{kite} shows the reconstructions with $L=9,15,21$. Obviously, both indicator functions can be used to give a quite good reconstruction of the kite shaped domain with the far field patterns at $9$ observation directions. The quality of the reconstruction of $\pa\Om$ (or $\Om$) can be improved greatly with the increase of the number $L$. As also shown in Figure \ref{annulus-L_L_25}, the indicator function $\mathcal{I}^{-}$ is good at determining the boundary $\pa\Om$.
Using the indicator function $\mathcal{I}^{+}$, we can not only reconstruct the source support $\Om$ but also approach the source function $f$ if the number $L$ is large enough.
\begin{figure}[htbp]
    \centering
    \begin{tabular}{ccc}
        \subfigure[$L=9$.]{
            \label{kite_L_9-I}
            \includegraphics[width=0.3\textwidth]{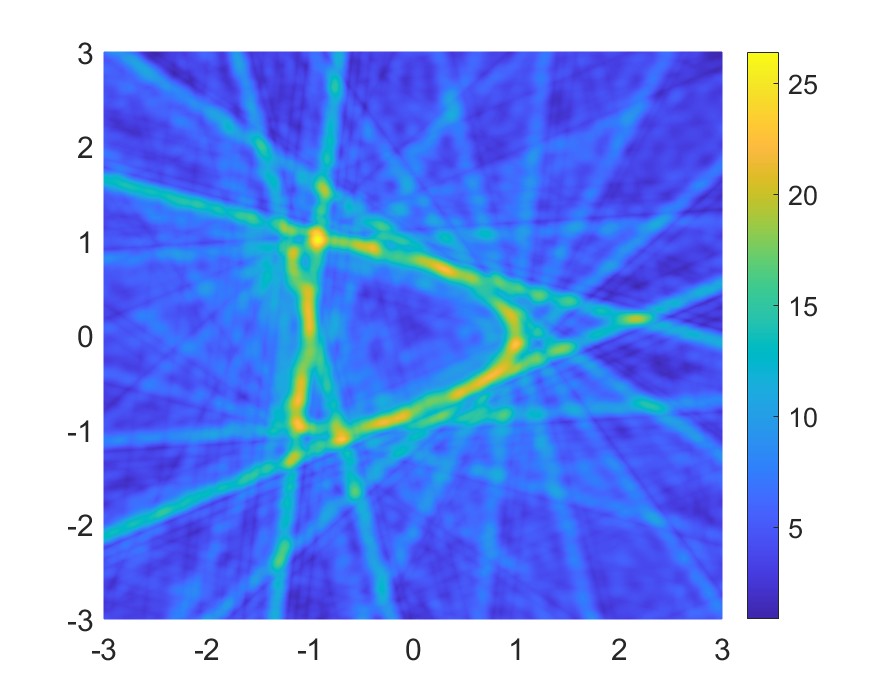}
        } &
        \subfigure[ $L=15$]{
            \label{kite_L_15-I}
            \includegraphics[width=0.3\textwidth]{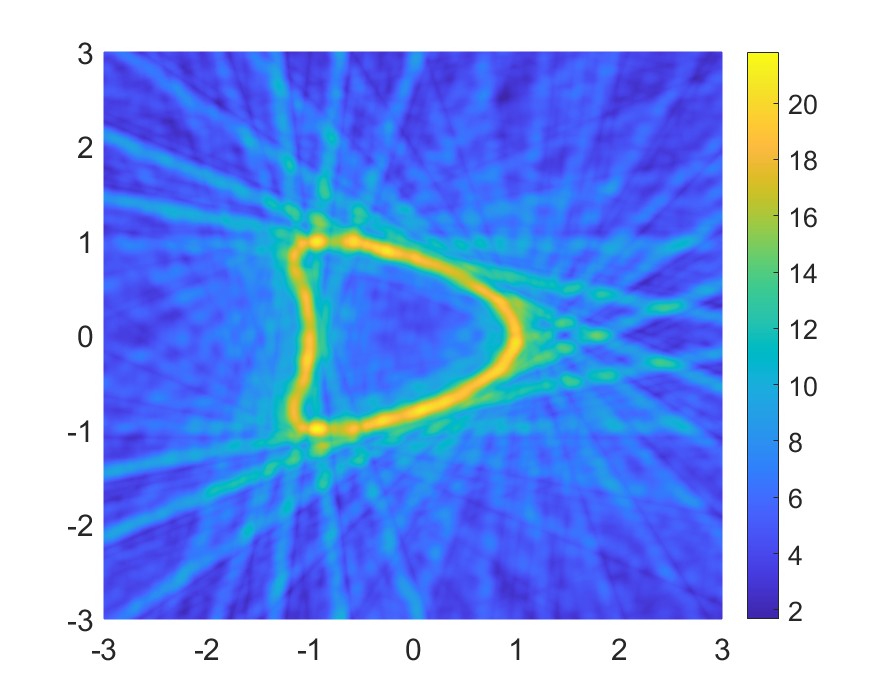}
        }&
         \subfigure[ $L=21$]{
            \label{kite_L_31-I}
            \includegraphics[width=0.3\textwidth]{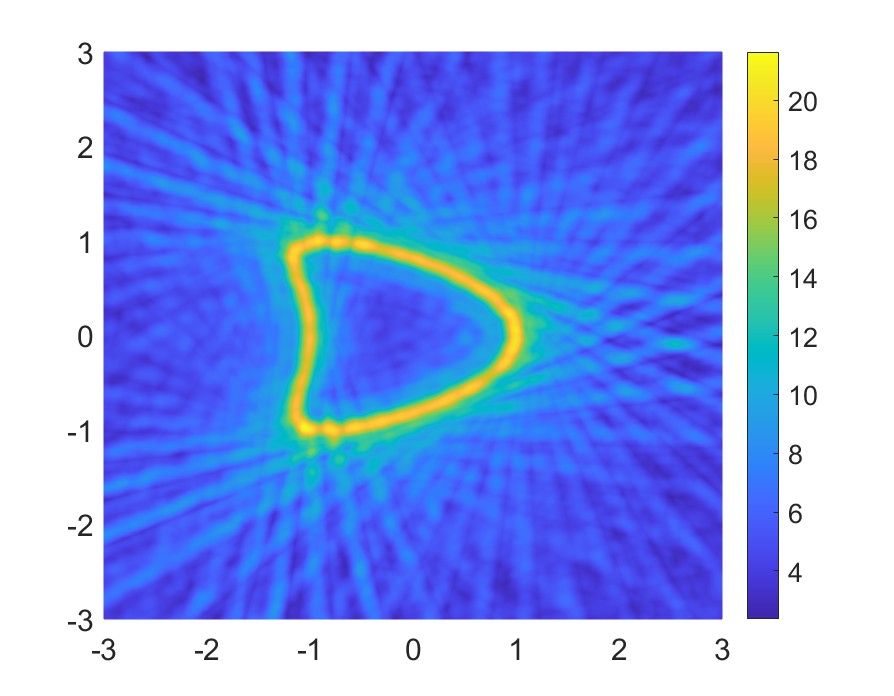}
            } \\
        \subfigure[$L=9$.]{
            \label{kite_L_9-PI}
            \includegraphics[width=0.3\textwidth]{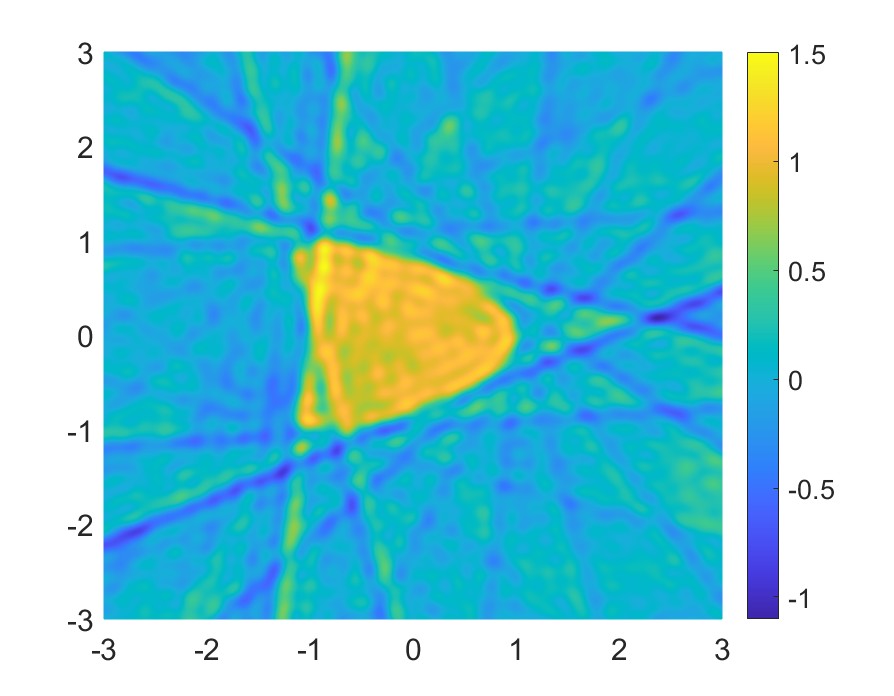}
        } &
        \subfigure[ $L=15$]{
            \label{kite_L_15-PI}
            \includegraphics[width=0.3\textwidth]{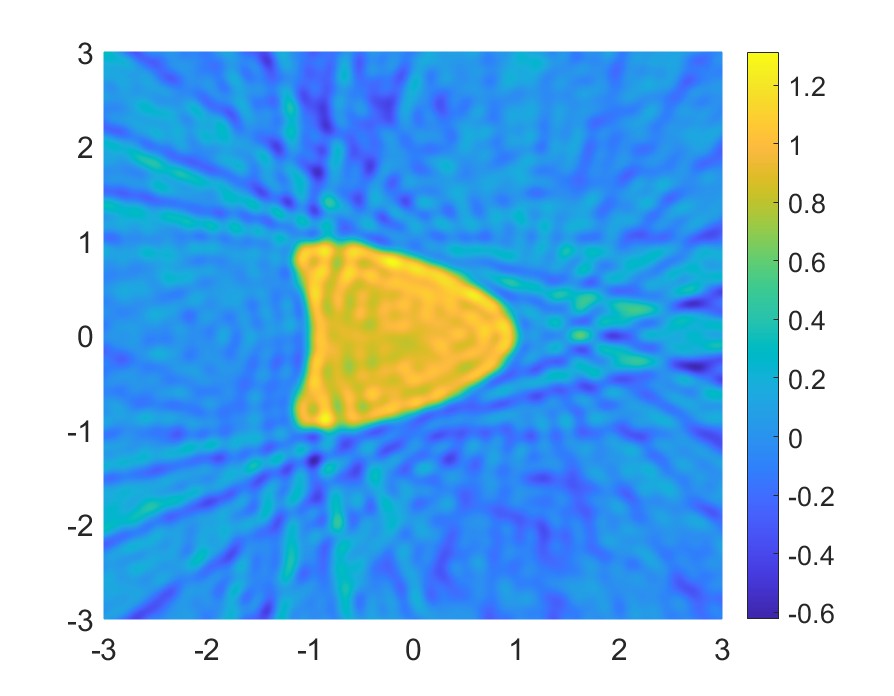}
        }&
         \subfigure[ $L=21$]{
            \label{kite_L_31-PI}
            \includegraphics[width=0.3\textwidth]{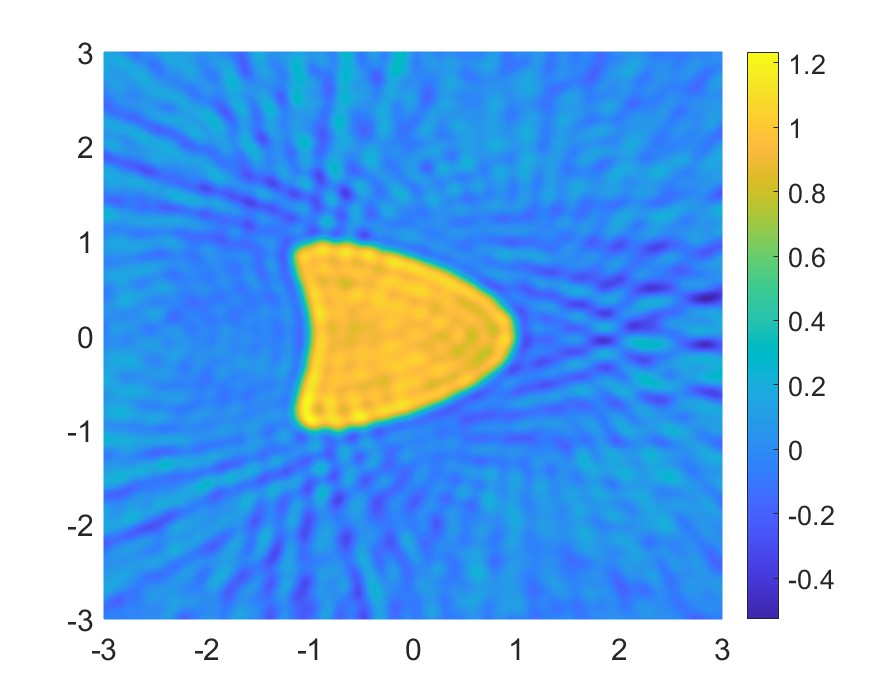}
        }
    \end{tabular}
    \caption{Reconstructions of a kite structured source with $\Lambda=30$. Top row: results by plotting $\mathcal{I}^-$. Bottom row: results by plotting  $\mathcal{I}^+$.}
    \label{kite}
\end{figure}

\subsection{Numerical experiments with various parameters}
This subsection is devoted to show the different effects of our method by varying $f,\Theta_L, \Lambda$ and noise. The source support $\Om$ is composed of four L-shaped polygons and an annulus. The true $\Om$ is shown in Figure \ref{source-mixed}. Unless otherwise stated, we will always take $f=\chi_{\Om}, L=15$, and $\Lambda=20$ with $\Omega=\Omega_{(1)}\cup\left(\Omega_{(2)}\cap\Omega_{(3)}^c\right)$, where
\begin{align*}
    \Omega_{(1)}&:=\{(y_1,y_2)\in\mathbb R^2|(y_1+1)^2+(y_2+1)^2\leq1\},\\
    \Omega_{(2)}&:=\{(y_1,y_2)\in\mathbb R^2|0\leq y_1\leq 1.5, 0\leq y_2\leq 1.5\},\\
    \Omega_{(3)}&:=\{(y_1,y_2)\in\mathbb R^2|0\leq y_1\leq 0.75, 0\leq y_2\leq 0.75\}.
\end{align*}
Figure \ref{source-mixed} shows the true shape of $\Omega$. 

The first direct sampling method for sources is proposed in \cite{spar-2020}, where the indicator function is defined as follows,
\be\label{I-alhs}
I_{ALHS}(z):=\sum\limits_{\hx\in\Theta_L}\left |\int_{k_1}^{k_{2\Lambda}}u^{\infty}(\hx, k)e^{ik\hx\cdot z}dk\right |, \quad z\in \R^n.
\en
For comparisons, we present the reconstruction of $f=\chi_{\Om}$ using $I_{ALHS}(z)$ in Figure \ref{old-inidcator}. Obviously, the shape and location of the source support $\Omega$ are roughly reconstructed with a low resolution.
We observe from Figures \ref{lambda-30-I} and \ref{lambda-30-PI}, the resolution is much improved by using the novel indicators  $\mathcal{I}^-$ and $\mathcal{I}^+$.  In particular, the indicator $\mathcal{I}^+$ gives a good reconstruction of the source function. 
\begin{figure}[htbp]
    \centering
    \subfigure[The support $\Omega$]{
            \label{source-mixed}
            \includegraphics[width=0.3\textwidth]{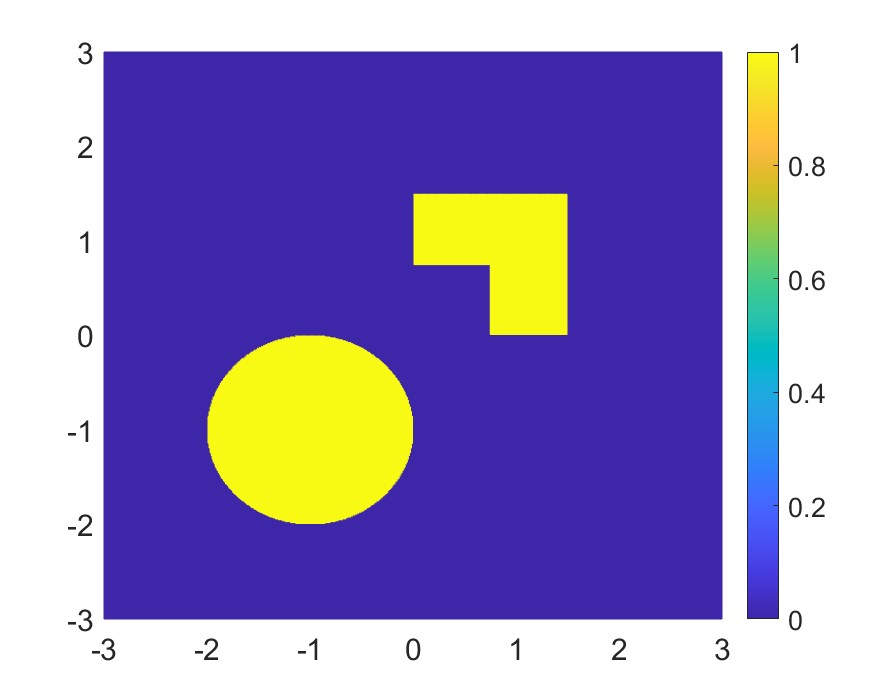}
        }
    \subfigure[ The reconstruction using the indicator \eqref{I-alhs} proposed in \cite{spar-2020}.]{
            \label{old-inidcator}
            \includegraphics[width=0.3\textwidth]{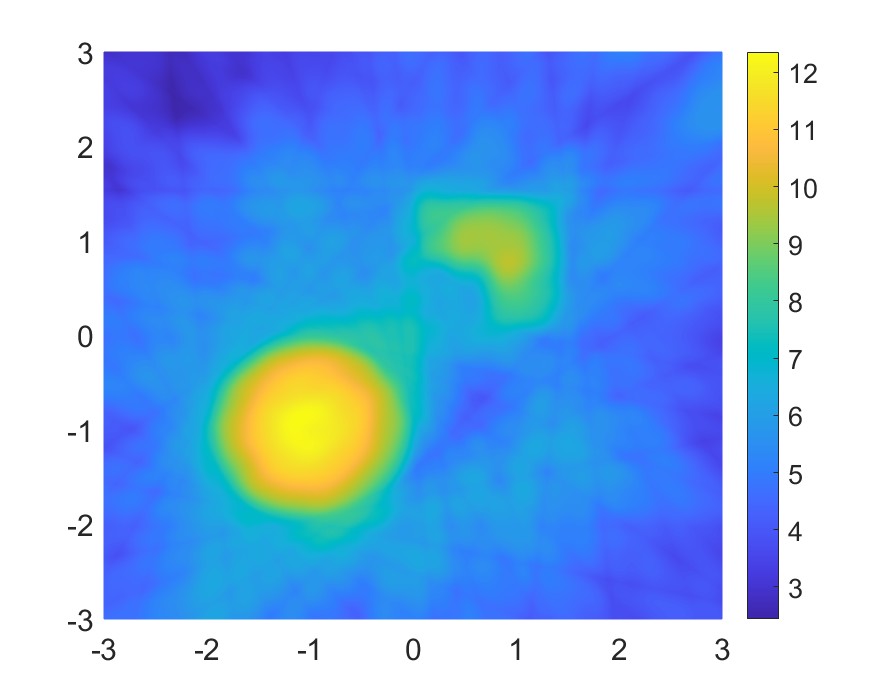}}
 \caption{Left: True source support $\Omega$ in this subsection. Right: Reconstruction by plotting $I_{ALHS}$ with $L=15,\Lambda=20$.}    \end{figure}

\begin{itemize}
    \item \textbf{Examples with different $\Lambda$.}
\end{itemize}
Figure \ref{lambda} displays the reconstructions for various $\Lambda$. Surprisingly, for such a complex structured support $\Om$, both the indicator
functions $\mathcal{I}^{\pm}$ give very satisfactory reconstructions with $\Lambda=20$. 
The quality of the reconstruction of $\Om$ can be improved by increasing $\Lambda$. We also observe that the second indicator function $\mathcal{I}^{+}$ performs better than the first indicator function $\mathcal{I}^{-}$. This fact will be verified in all the subsequent numerical examples.

\begin{figure}[htbp]
   \centering
    \begin{tabular}{ccc}
        \subfigure[$\Lambda=20$.]{
            \label{lambda-30-I}
        \includegraphics[width=0.3\textwidth]{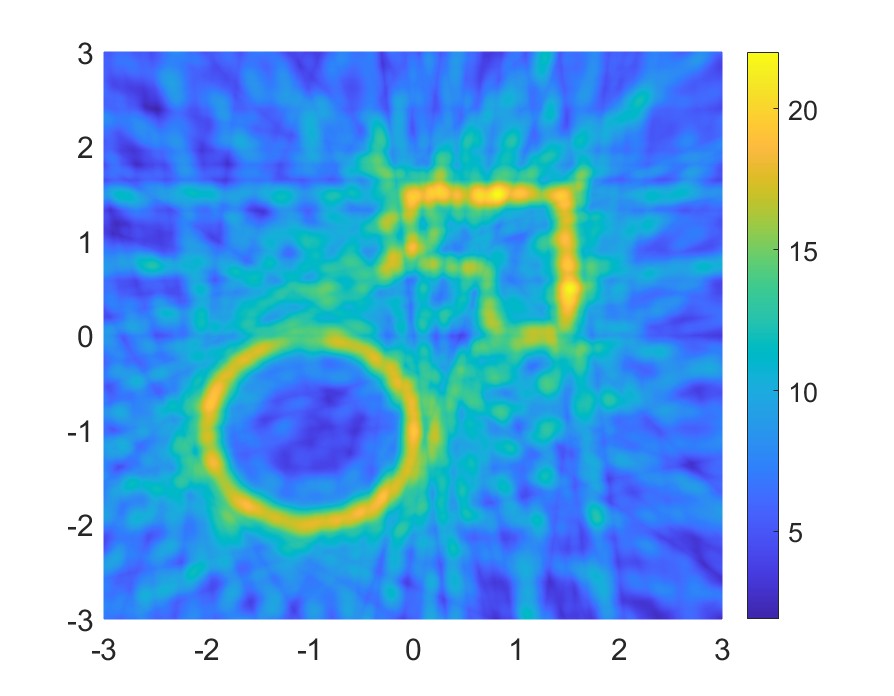}
        }&
        \subfigure[ $\Lambda=30$.]{
            \label{lambda-40-I}
            \includegraphics[width=0.3\textwidth]{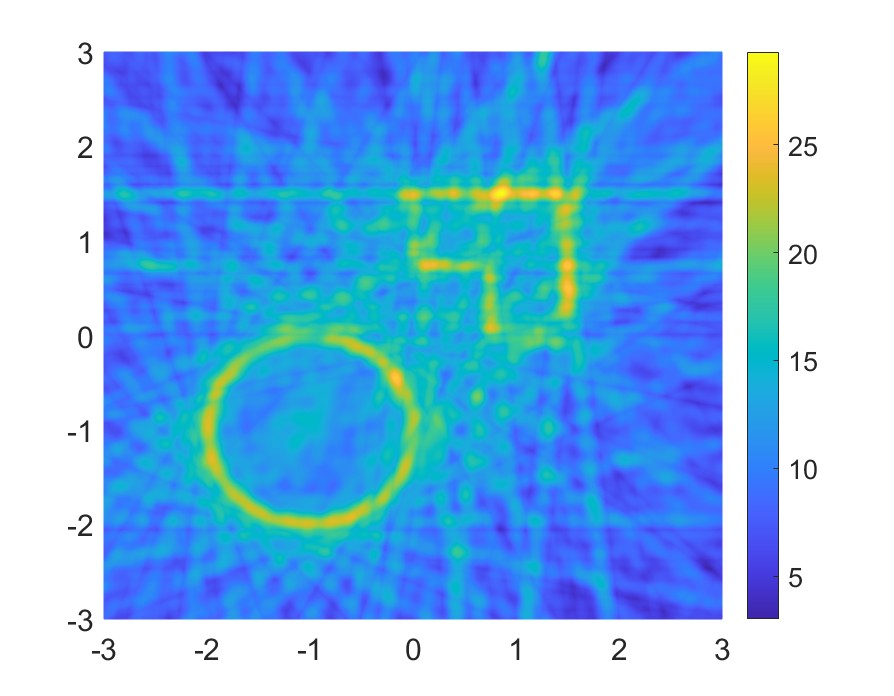}
        }&
         \subfigure[$\Lambda=40$.]{
            \label{lambda-50-I}
            \includegraphics[width=0.3\textwidth]{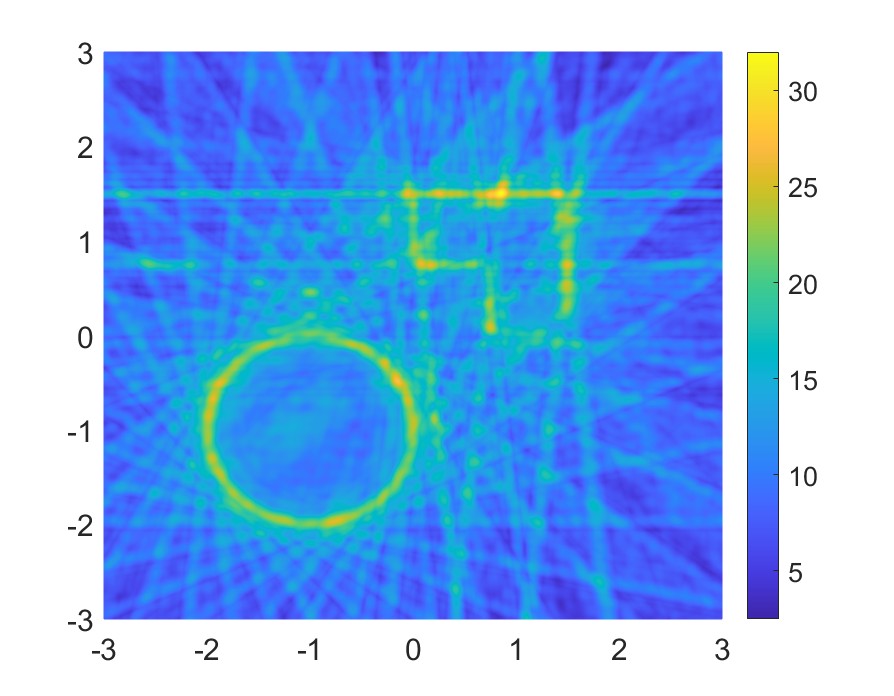}
        }\\
        \subfigure[$\Lambda=20$.]{
            \label{lambda-30-PI}
            \includegraphics[width=0.3\textwidth]{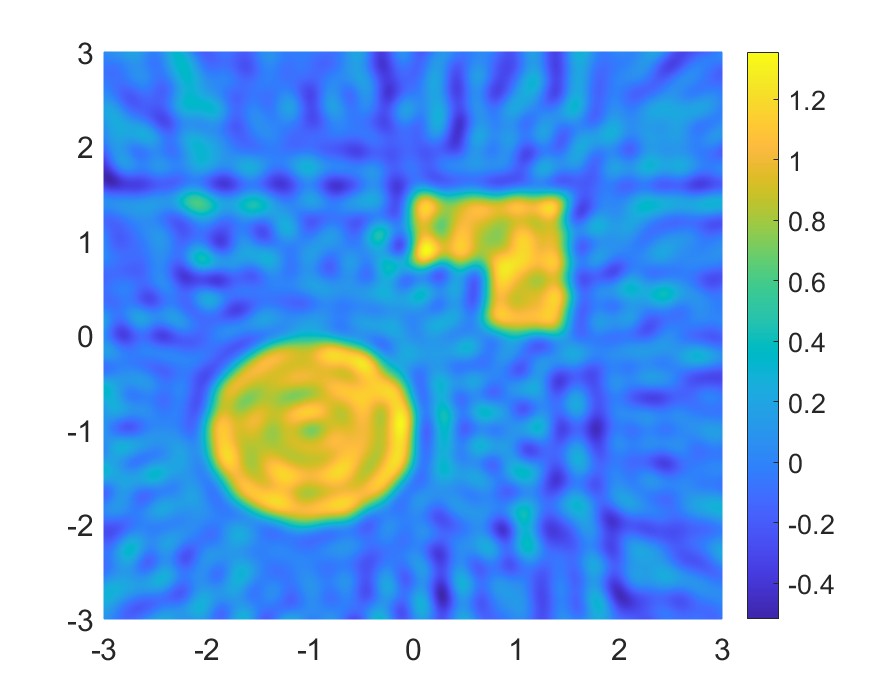}
        }&
        \subfigure[ $\Lambda=30$.]{
            \label{lambda-40-PI}
            \includegraphics[width=0.3\textwidth]{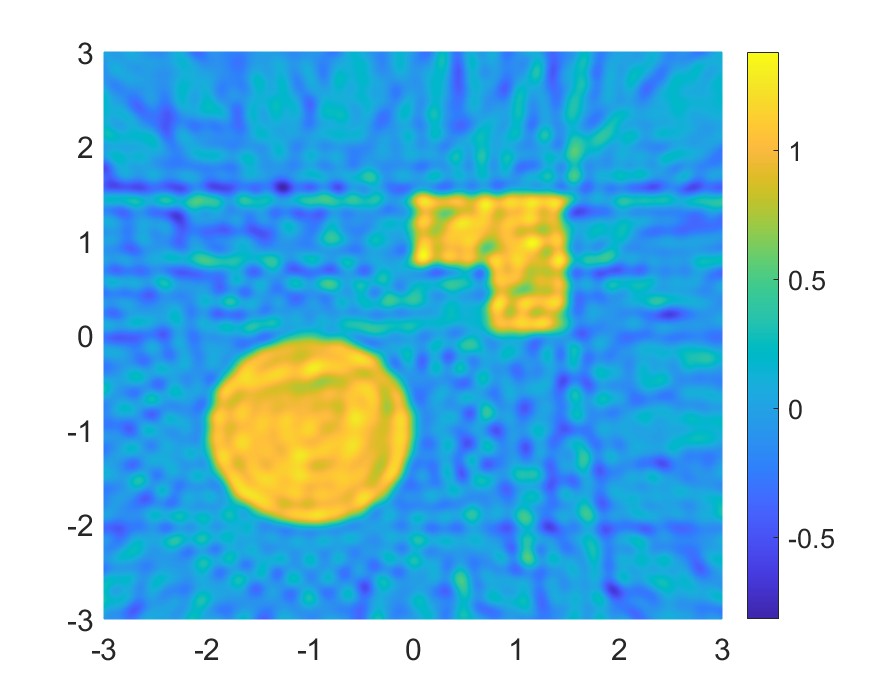}
        }&
         \subfigure[ $\Lambda=40$.]{
            \label{lambda-50-PI}
            \includegraphics[width=0.3\textwidth]{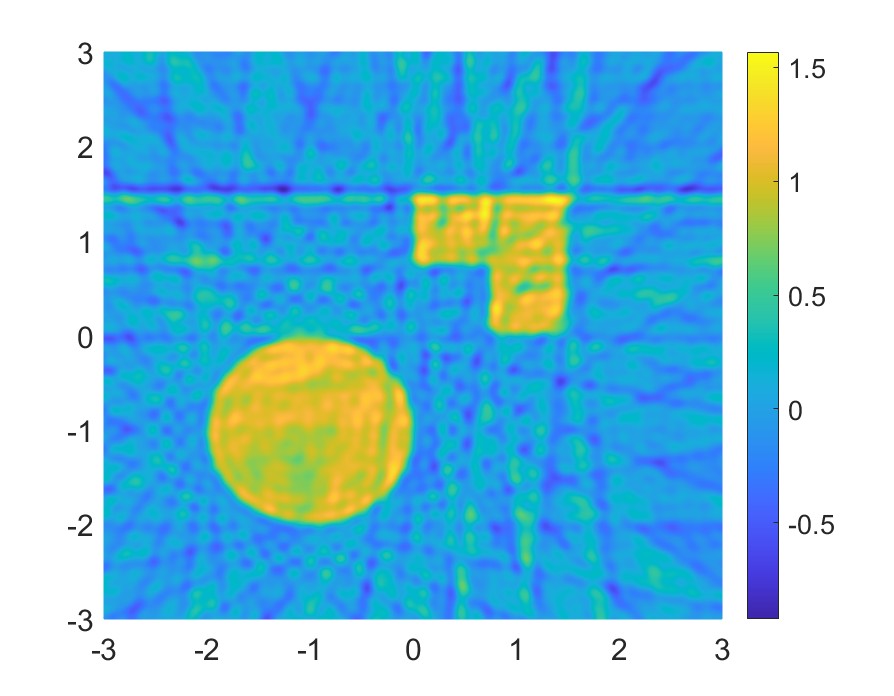}
        }
    \end{tabular}
    \caption{Reconstructions for various $\Lambda$. Top row: reconstructions by plotting $\mathcal{I}^-$. Bottom row: reconstructions by plotting $\mathcal{I}^+$.}
    \label{lambda}
\end{figure}

\begin{itemize}
    \item \textbf{Examples with different $\Theta_L$.}
\end{itemize}
Figure \ref{ThetaL} indicates that the fidelity of the reconstruction improves as $L$ increases.  Note that $L=31$ is the minimum value required in \eqref{Lin3.4} for uniquely determine all the corners of $\pa\Omega$ and the annulus. Numerically, we have already obtained quite a good reconstruction of $\pa\Om$ with $L=15$. The same as shown in Figure \ref{kite} for the kite shaped domain, the quality of the reconstruction of $\Om$ can be improved with the increase of $L$.
\begin{figure}[htbp]
\centering
    \begin{tabular}{ccc}
        \subfigure[$L=15$.]{
            \label{ThetaL-31-I}
            \includegraphics[width=0.3\textwidth]{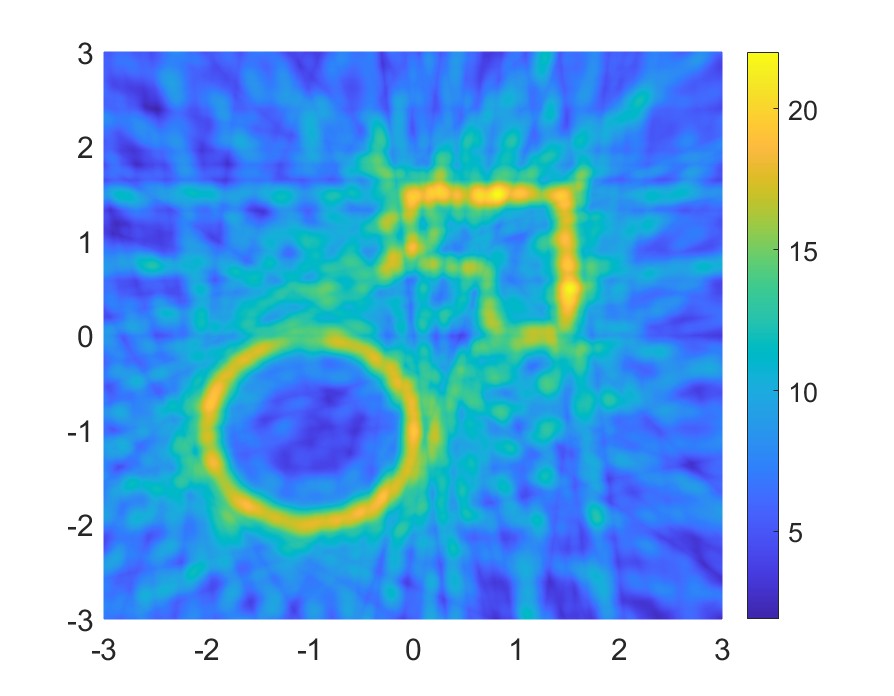}
        }&
        \subfigure[$L=23$.]{
            \label{ThetaL-41-I}
            \includegraphics[width=0.3\textwidth]{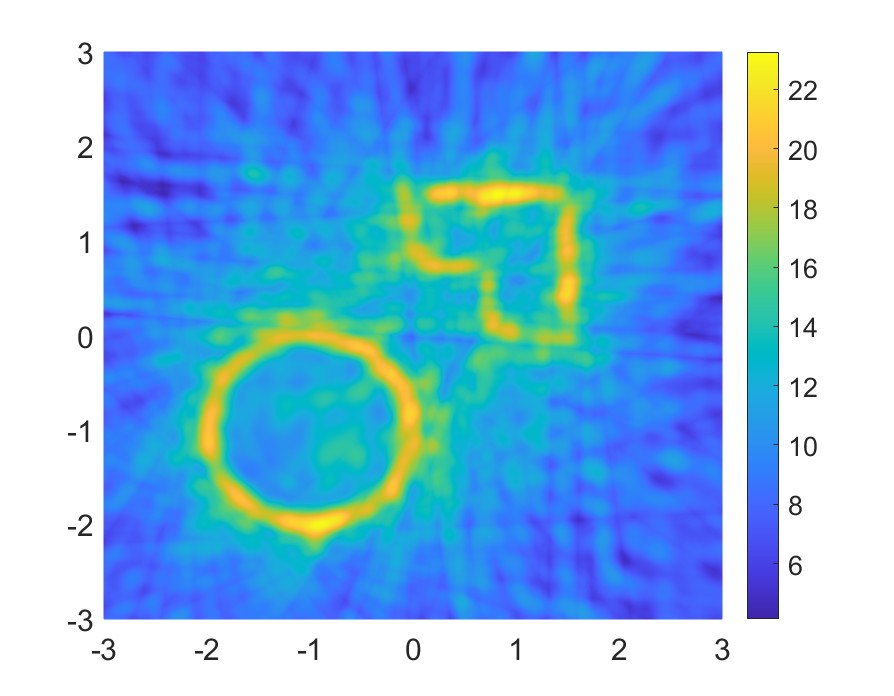}
        }&
         \subfigure[$L=31$.]{
            \label{Theta-51-I}
            \includegraphics[width=0.3\textwidth]{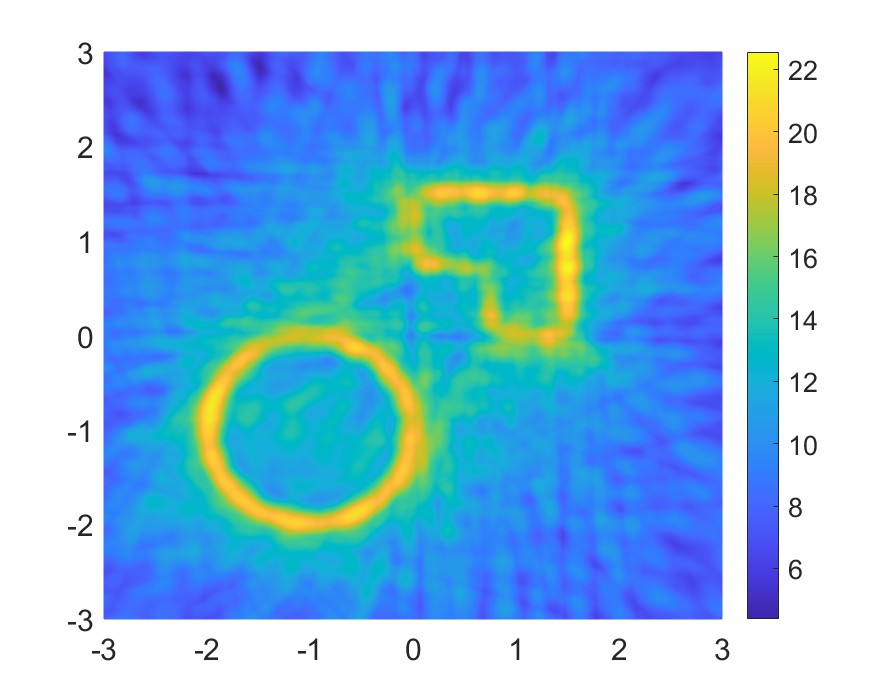}
        }\\
                \subfigure[$L=15$.]{
            \label{ThetaL-31-PI}
            \includegraphics[width=0.3\textwidth]{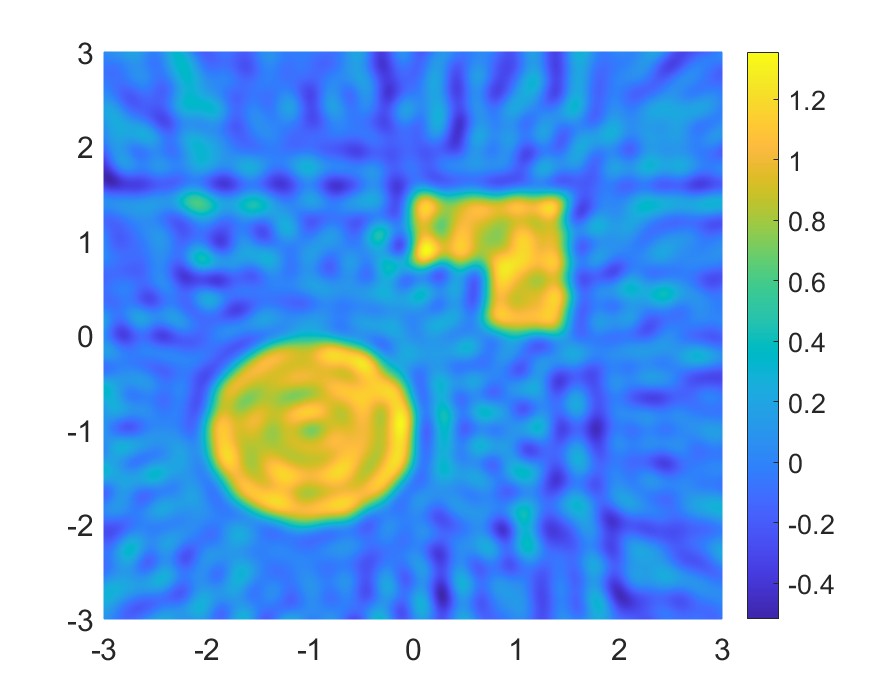}
        }&
        \subfigure[$L=23$.]{
            \label{ThetaL-41-PI}
            \includegraphics[width=0.3\textwidth]{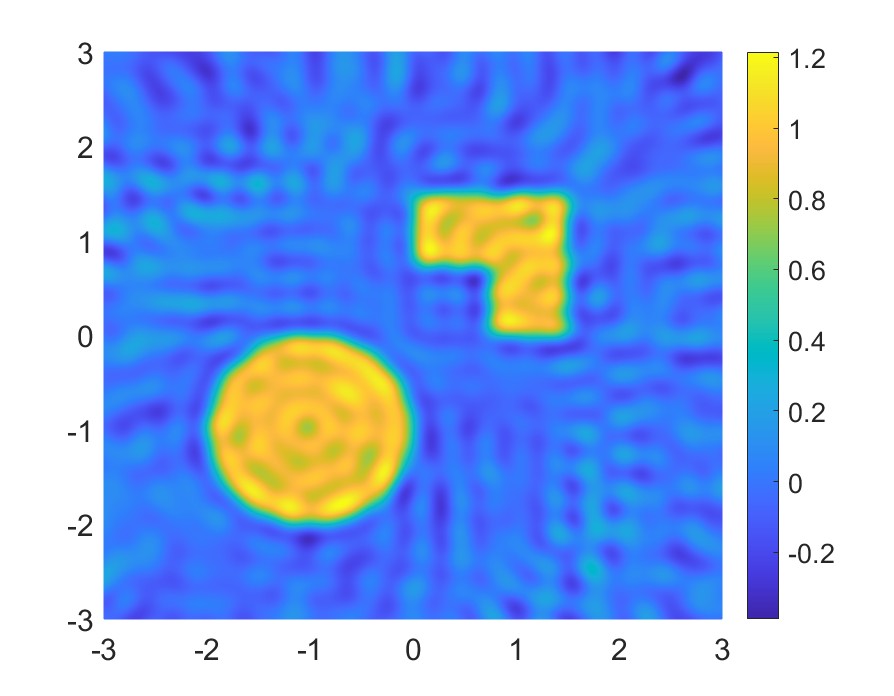}
        }&
         \subfigure[$L=31$.]{
            \label{Theta-51-PI}
            \includegraphics[width=0.3\textwidth]{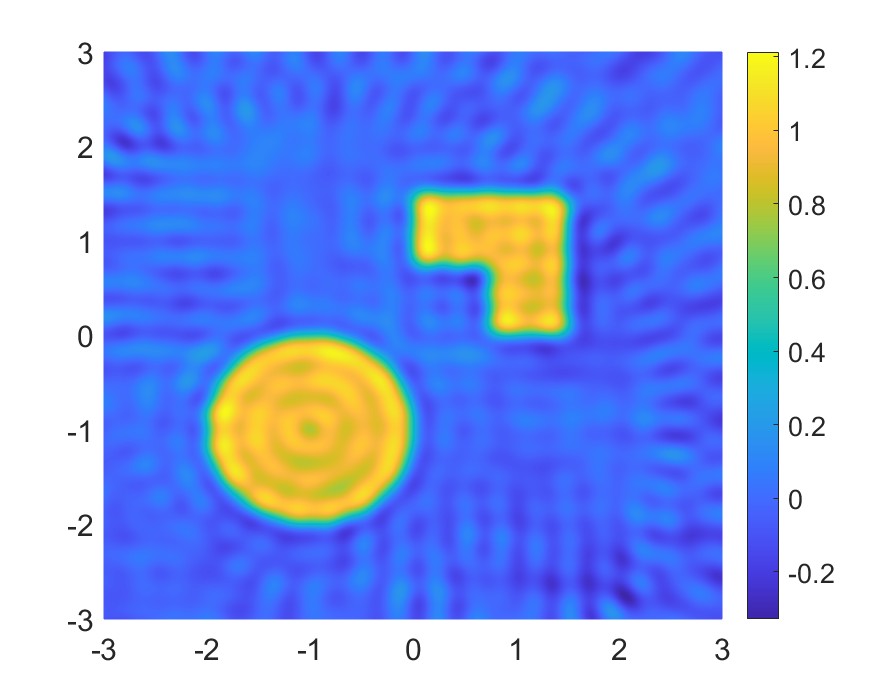}
        }
    \end{tabular}
    \caption{Reconstructions for various $L$. Top row: reconstructions by plotting $\mathcal{I}^-$. Bottom row: reconstructions by plotting $\mathcal{I}^+$.}
    \label{ThetaL}
\end{figure}

For a fixed $L$, we also consider the limited aperture case with the observation direction set
\ben
\Theta_L^\gamma:=\left\{\Big(\cos\frac{(\gamma l-0.7L) \pi}{L}, \sin\frac{(\gamma l-0.7L) \pi}{L}\Big)\in\mathbb{S}^1\,|\,0\leq l\leq L-1\right\},\quad 0<\gamma\leq 2.
\enn
In the full aperture case,  $\gamma=2$. The aperture of $\Theta_L^\gamma$ decreases as $\gamma>0$ decreases. 
Figure \ref{limited-aperture} presents the results with $\gamma=0.7, 0.5, 0.3$. 
Numerically, the limited aperture data indeed present a severe challenge for the constructions. The annulus is elongated along the horizontal direction. In particular, the elongation increases as $\gamma$ decreases. This is reasonable because the observation directions focus on the vertical direction as $\gamma$ decreases. We also observe that all the horizontal edges are well reconstructed even the aperture is very limited ($\gamma=0.3$).
\begin{figure}[htbp]
\centering
    \begin{tabular}{ccc}
        \subfigure[$\gamma=0.7$.]{
            \label{07pi-I}
            \includegraphics[width=0.3\textwidth]{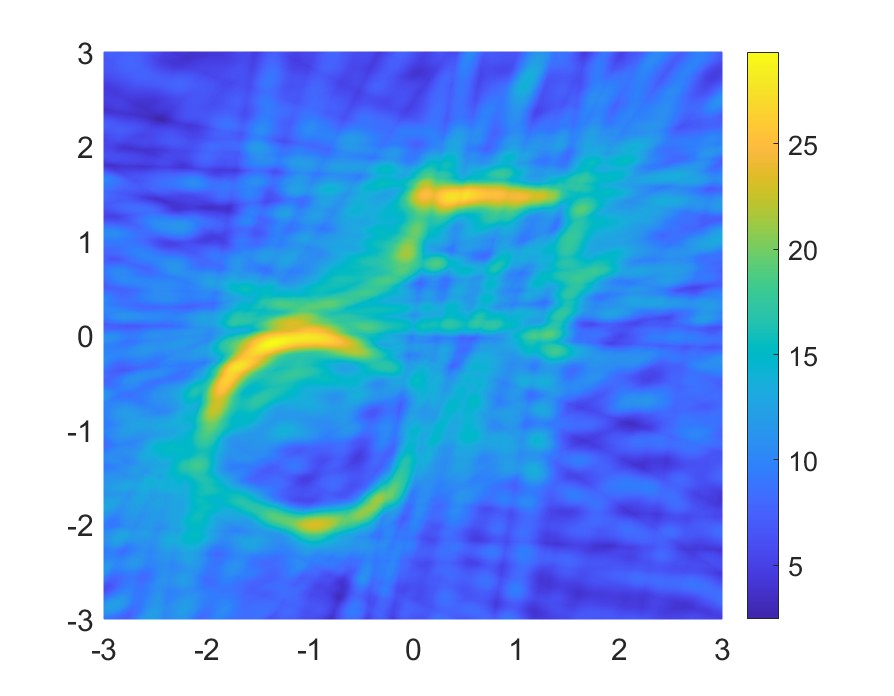}
        }&
        \subfigure[$\gamma=0.5$.]{
            \label{05pi-I}
            \includegraphics[width=0.3\textwidth]{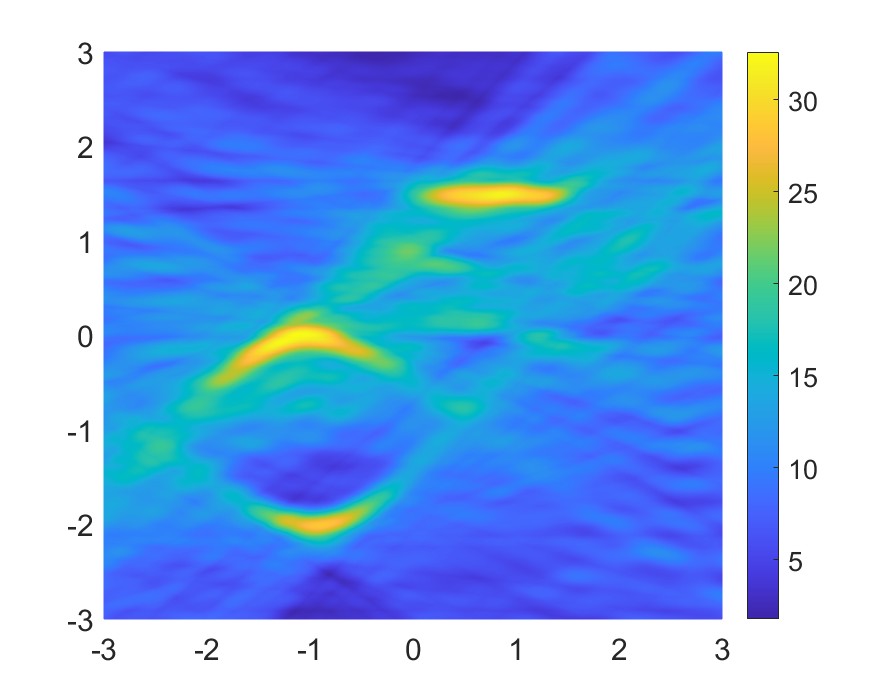}
        }&
         \subfigure[$\gamma=0.3$.]{
            \label{03pi-I}
            \includegraphics[width=0.3\textwidth]{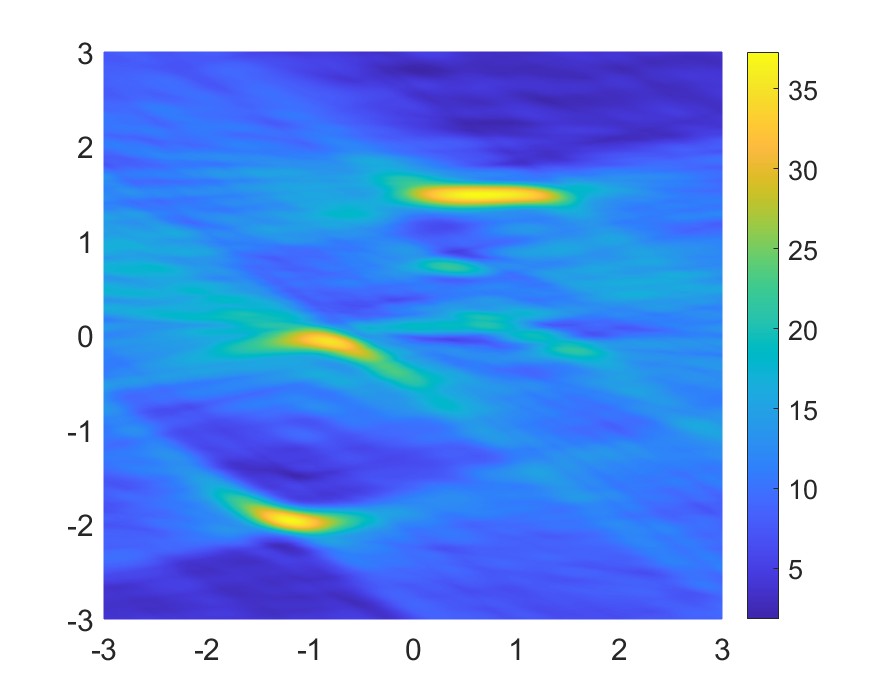}
        }\\
          \subfigure[$\gamma=0.7$.]{
            \label{07pi-PI}
            \includegraphics[width=0.3\textwidth]{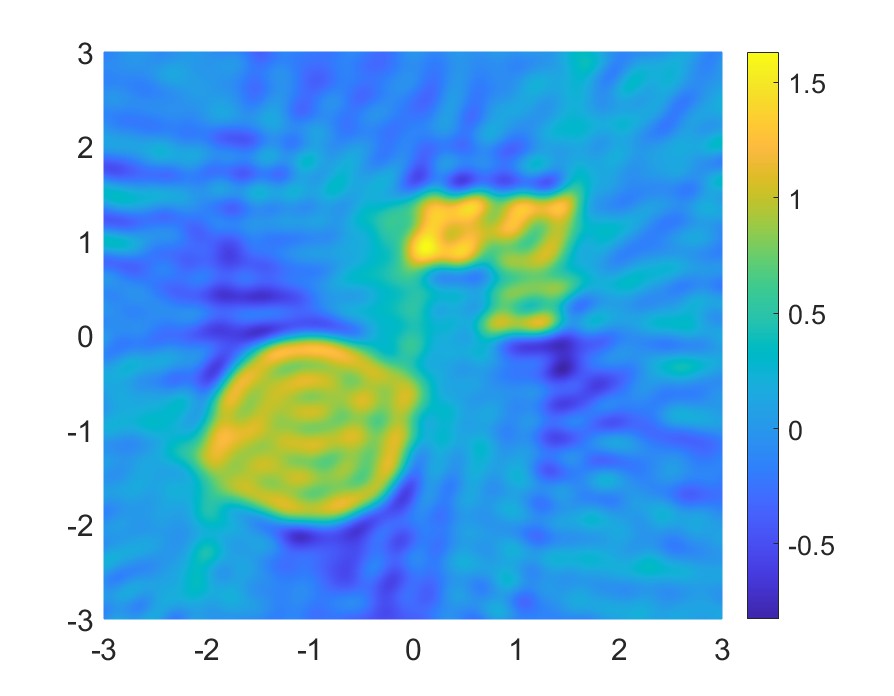}
        }&
        \subfigure[$\gamma=0.5$.]{
            \label{05pi-PI}
            \includegraphics[width=0.3\textwidth]{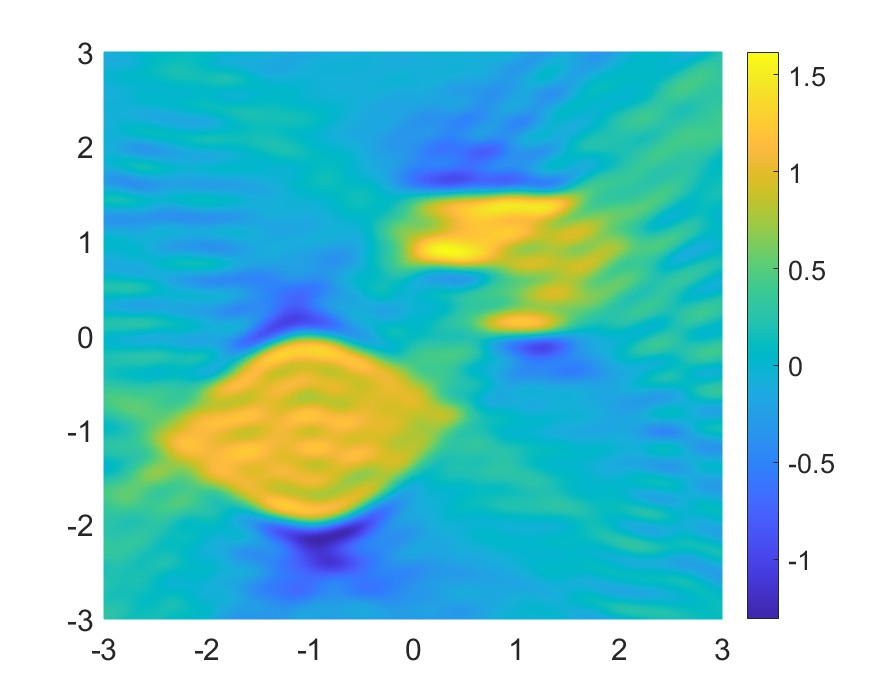}
        }&
         \subfigure[$\gamma=0.3$.]{
            \label{03pi-PI}
            \includegraphics[width=0.3\textwidth]{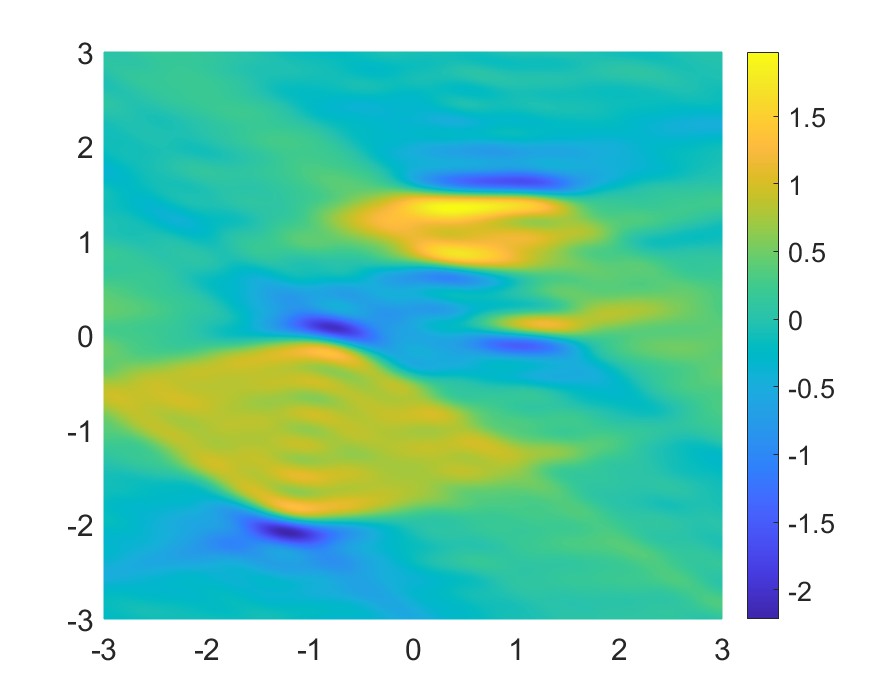}
        }
    \end{tabular}
    \caption{Reconstructions for various $\Theta_L$ (different $\gamma$). Top row: reconstructions by plotting $\mathcal{I}^-$. Bottom row: reconstructions by plotting $\mathcal{I}^+$.}
    \label{limited-aperture}
\end{figure}

\begin{itemize}
    \item \textbf{Examples with different $f$.}
\end{itemize}
Figure \ref{f} shows the reconstructions for three different source functions 
\ben
f_1(y)=e^{-0.05|y|^2}, \quad f_2(y)=1 \quad\mbox{and}\quad f_3(y)=e^{0.05|y|^2},\quad y\in\Omega.
\enn
Obviously, the quality of the reconstruction of $\Om$ by using $\mathcal{I}^-$ depends on the value of the source function $f$. The point with a larger value of $f$ has a better reconstruction.
The second indicator $\mathcal{I}^+$ can not only stably reconstruct the source support $\Om$ but also reflect the value of the source function $f$.
\begin{figure}[htbp]
    \centering 
    \begin{tabular}{ccc}
        \subfigure[$f_1(y)=e^{-0.25|y|^2},\ y\in\Omega$.]{
            \label{f-neg-I}
            \includegraphics[width=0.3\textwidth]{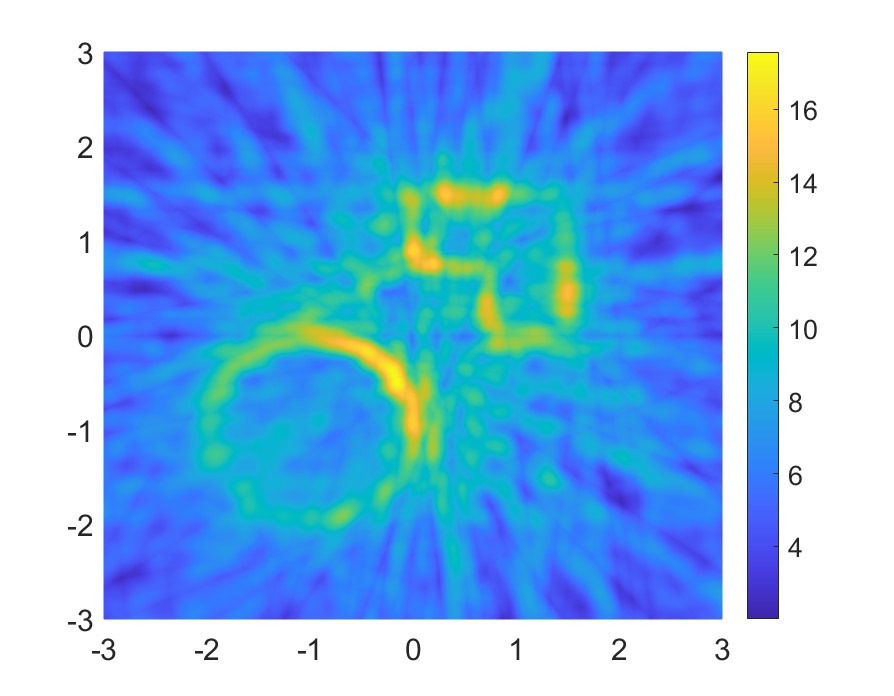}
        }&
        \subfigure[ $f_2(y)=1,\ y\in\Omega$]{
            \label{f-0-I}
            \includegraphics[width=0.3\textwidth]{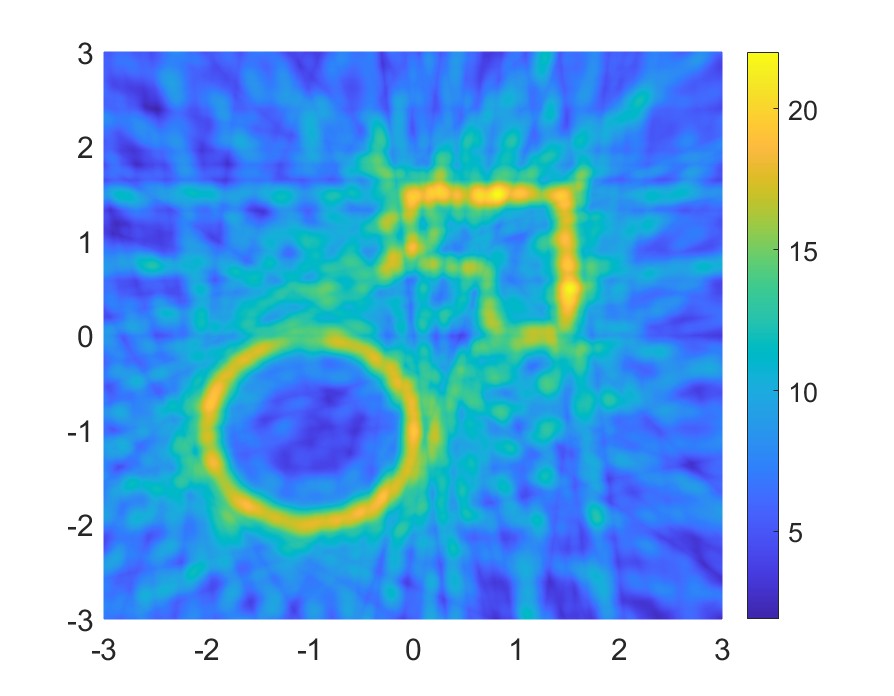}
        }&
         \subfigure[$f_3(y)=e^{0.25|y|^2},\ y\in\Omega$]{
            \label{f-pos-I}
            \includegraphics[width=0.3\textwidth]{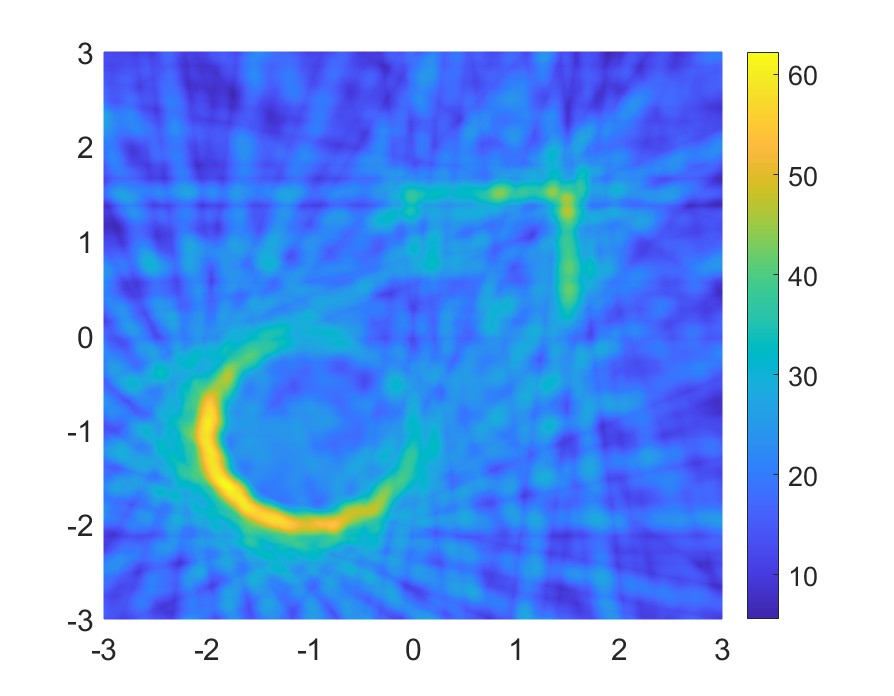}
        }\\
        \subfigure[$f_1(y)=e^{-0.25|y|^2},\ y\in\Omega$.]{
            \label{f-neg-PI}
            \includegraphics[width=0.3\textwidth]{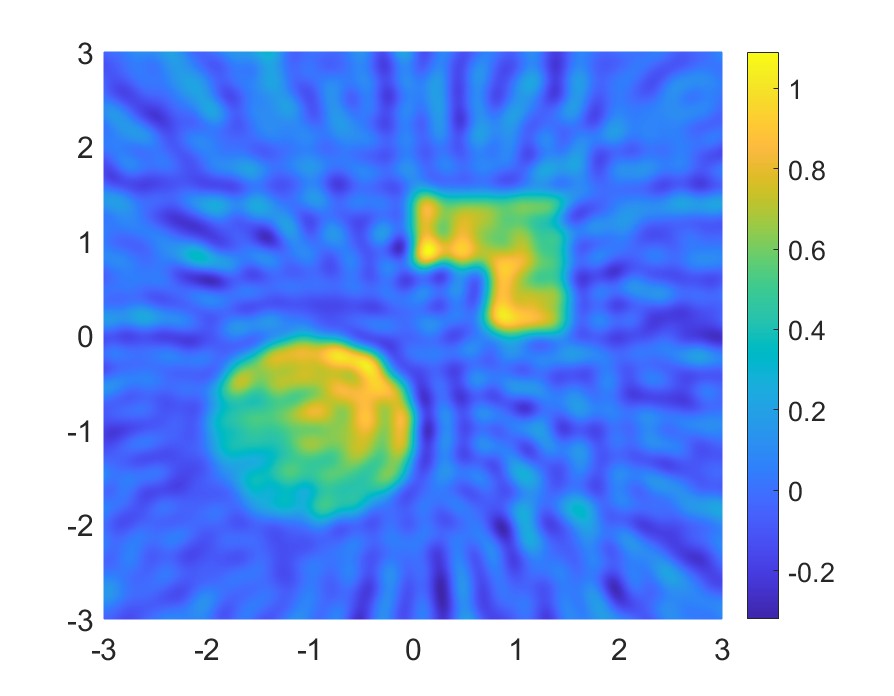}
        }&
        \subfigure[ $f_2(y)=1,\ y\in\Omega$]{
            \label{f-0-PI}
            \includegraphics[width=0.3\textwidth]{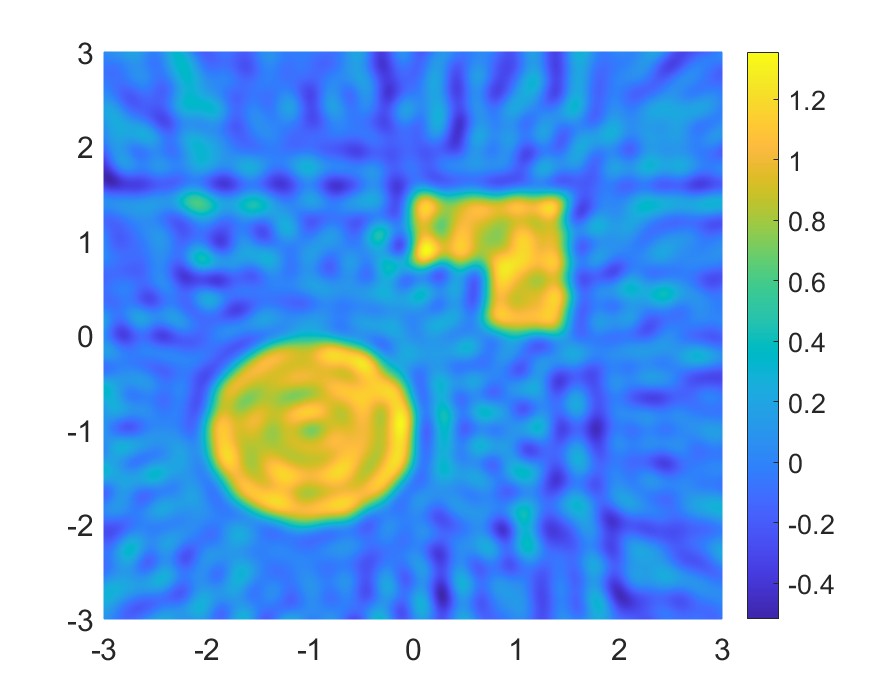}
        }&
         \subfigure[$f_3(y)=e^{0.25|y|^2},\ y\in\Omega$]{
            \label{f-pos-PI}
            \includegraphics[width=0.3\textwidth]{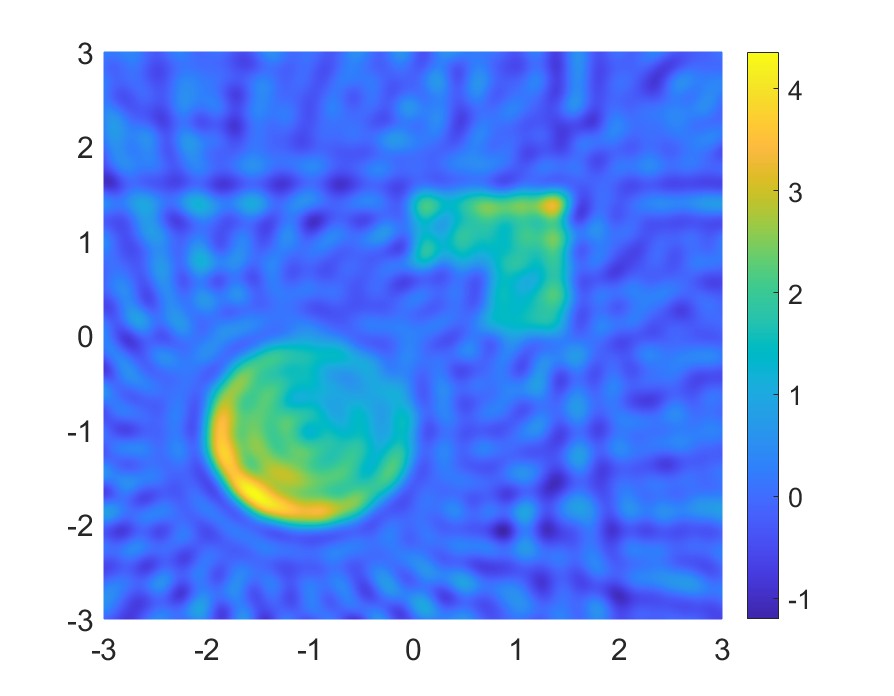}
        }
    \end{tabular}
    \caption{Reconstructions for various $f$. Top row: reconstructions by plotting $\mathcal{I}^-$. Bottom row: reconstructions by plotting $\mathcal{I}^+$.}
    \label{f}
\end{figure}

\begin{itemize}
    \item \textbf{Examples with systematic errors.}
\end{itemize}
In this example, we investigate the influence of systematic errors on our methods. We represent systematic errors by independent random variables with non-zero mean. Precisely, we set
\begin{align*}
     u^{\infty,\delta}_{sys}(\hx_l,k_m)=u^{\infty,\delta}(\hx_l,k_m)+\big[\Tilde{X}_{l,m}+\Tilde{Y}_{l,m}i\big],
 \end{align*}
 where $\Tilde{X}_{l,m},\Tilde{Y}_{l,m}\sim N(\mu,\sigma^2)$ and $N(\mu,\sigma^2)$ is a Gaussian distribution with mean $\mu\neq 0$ and variance $\sigma^2$.

 Figures \ref{mu0.01sig0.05-m} and \ref{mu0.01sig0.05-p} show that the quality of the reconstructions for $\mathcal{I}^{\pm}$ decreases after adding systematic errors with $\mu=0.01$ and $\sigma=0.05$. Note that these $\mu$ and $\sigma$ are no longer small, they are on the same order of magnitude as some far-field data. Unfortunately, our methods failed when $\mu=\sigma=0.1$, systematic errors have indeed caused us trouble, details can be seen in Figures \ref{mu0.1sig0.1-m} and \ref{mu0.1sig0.1-p}. Furthermore, Figures \ref{mo-mu0.1sig0.1-m} and \ref{mo-mu0.1sig0.1-p} illustrate that if we have prior knowledge of $\mu$, the reconstruction can be improved by the following modified indicators $\mathcal{I}^{\pm}_M$,
\begin{align*}
    \mathcal{I}^+_M(z)&:=\frac{1}{4\pi L}\sum\limits_{\hx\in\Theta_L}\int^{k_{2\Lambda}}_{k_{1}}k\left[(u^{\infty,\delta}_{sys}(\hat{x},k)-\mu)e^{ik\hx\cdot z}+(u^{\infty,\delta}_{sys}(-\hat{x},k)-\mu)e^{-ik\hx\cdot z}\right]dk,\quad z\in\mathcal{G},\\
\mathcal{I}^-_M(z)&:=\frac{1}{L}\sum\limits_{\hx\in\Theta_L}\left|\int^{k_{2\Lambda}}_{k_{1}}k\left[(u^{\infty,\delta}_{sys}(\hat{x},k)-\mu)e^{ik\hx\cdot z}-(u^{\infty,\delta}_{sys}(-\hat{x},k)-\mu)e^{-ik\hx\cdot z}\right]dk\right|,\quad z\in\mathcal{G}.
\end{align*}
It is easy to derive that 
\begin{align*}
    \mathbb E[\mathcal{I}^+(z)-\mathcal{I}^+_M(z)]=0\quad {\rm and }\quad\mathbb E[\mathcal{I}^-(z)-\mathcal{I}^-_M(z)]=C(\sigma),\ \forall\ z\in\mathcal{G},
\end{align*}
here $C(\sigma)$ is a constant depends on $\sigma$. This explains why the modified indicators $\mathcal{I}^{\pm}_M$ work.

\begin{figure}[htbp]
\centering
    \begin{tabular}{ccc}
        \subfigure[Result by plotting $\mathcal{I}^-$.]{
            \label{mu0.01sig0.05-m}
            \includegraphics[width=0.3\textwidth]{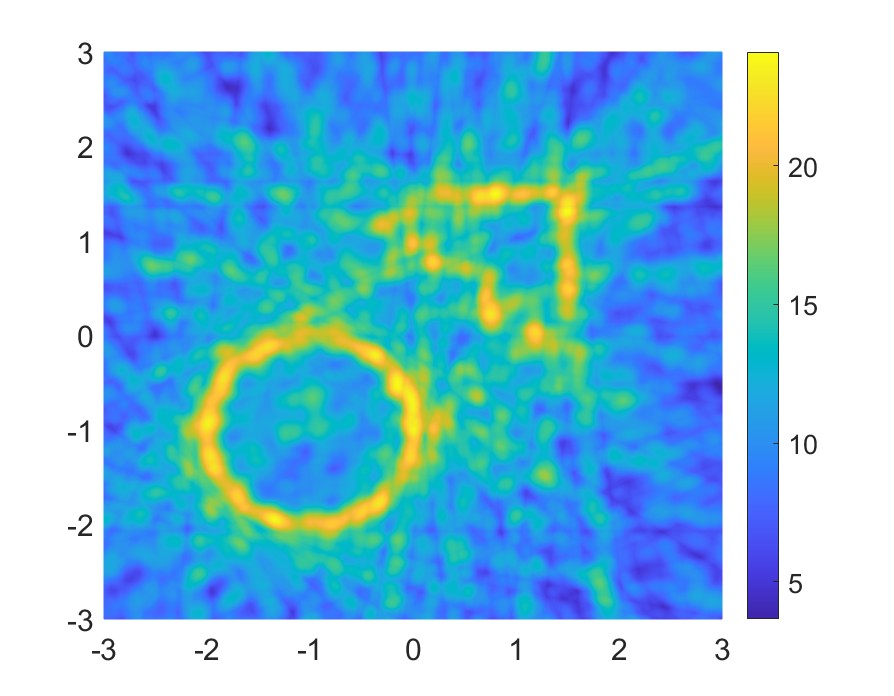}
        }&
        \subfigure[Result by plotting $\mathcal{I}^-$.]{
            \label{mu0.1sig0.1-m}
            \includegraphics[width=0.3\textwidth]{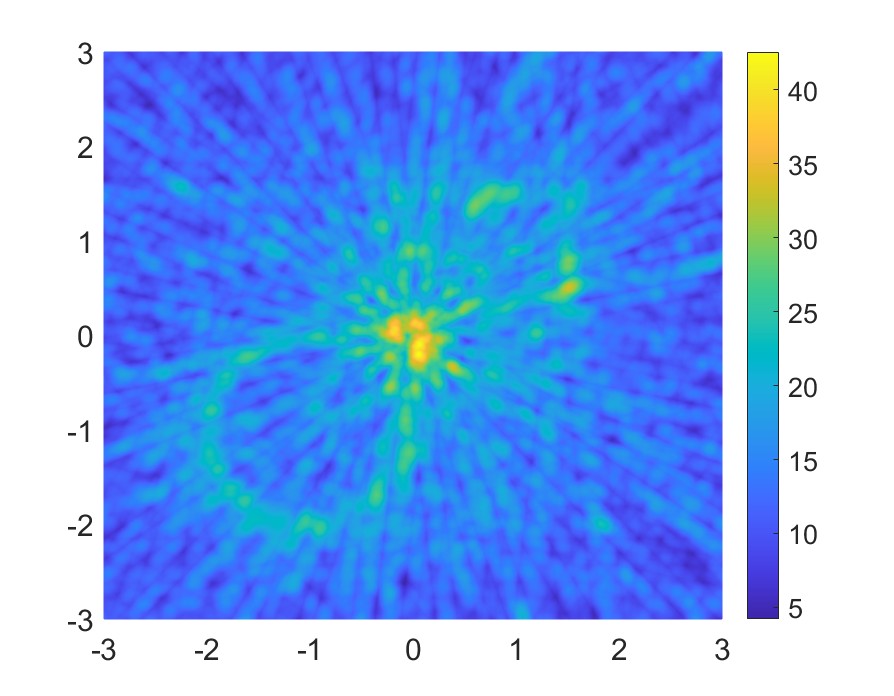}
        }&
         \subfigure[Result by plotting $\mathcal{I}^-_M$.]{
            \label{mo-mu0.1sig0.1-m}
            \includegraphics[width=0.3\textwidth]{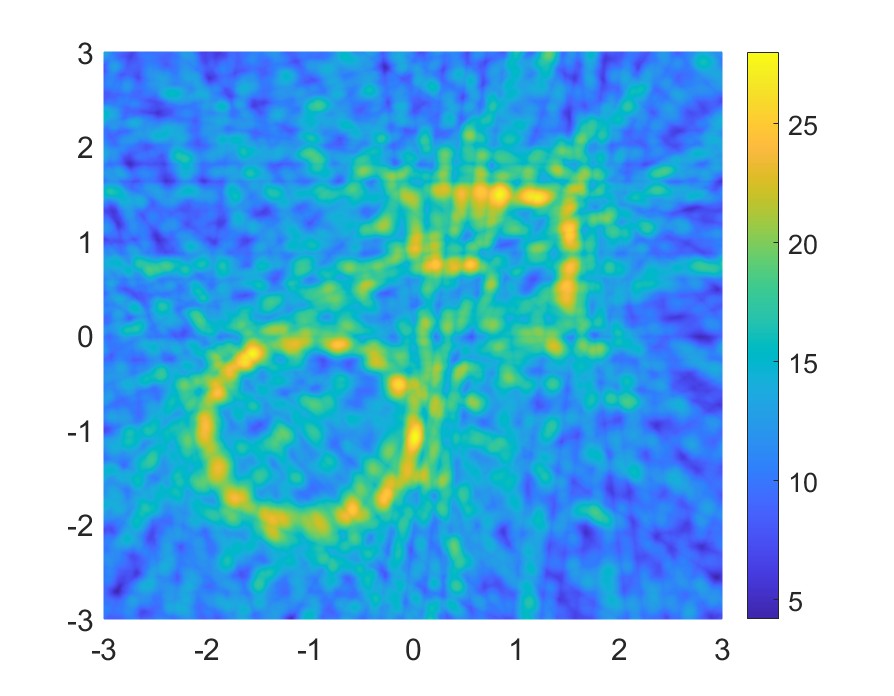}
        }\\
          \subfigure[Result by plotting $\mathcal{I}^+$.]{
            \label{mu0.01sig0.05-p}
            \includegraphics[width=0.3\textwidth]{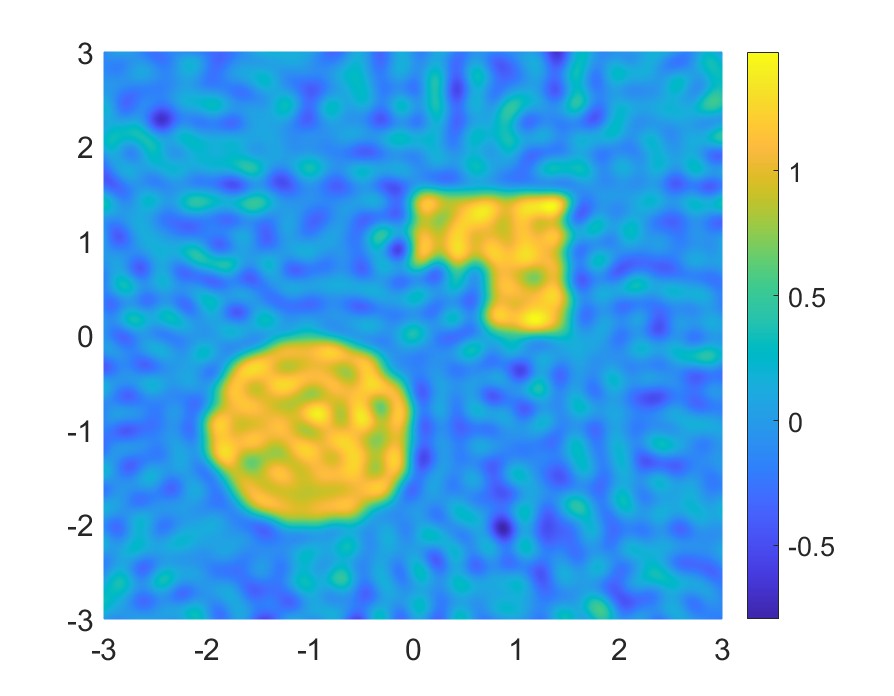}
        }&
        \subfigure[Result by plotting $\mathcal{I}^+$.]{
            \label{mu0.1sig0.1-p}
            \includegraphics[width=0.3\textwidth]{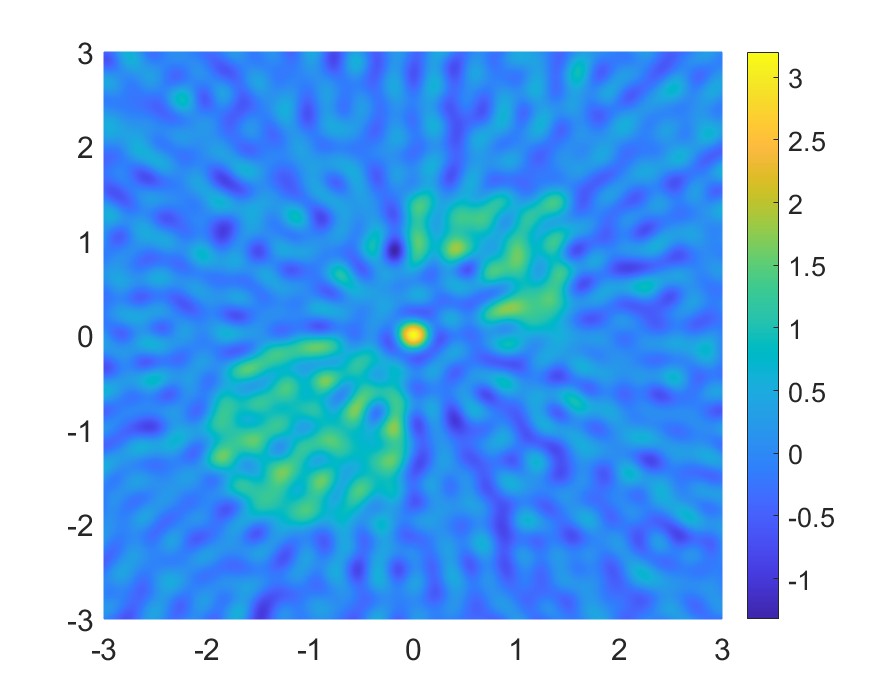}
        }&
         \subfigure[Result by plotting $\mathcal{I}^+_M$.]{
            \label{mo-mu0.1sig0.1-p}
            \includegraphics[width=0.3\textwidth]{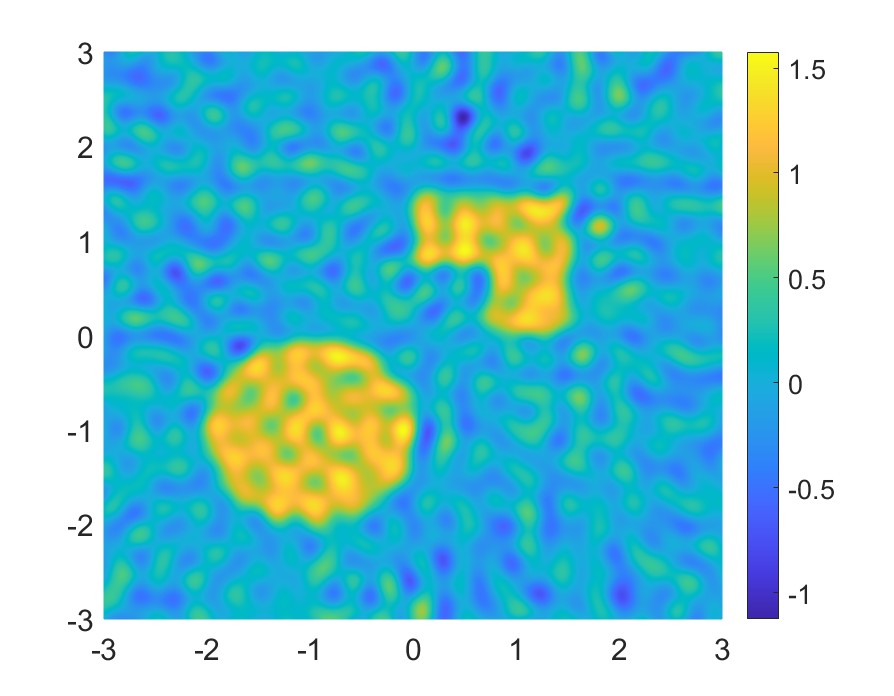}
        }
    \end{tabular}
    \caption{Reconstructions for various systematic errors. Left: $\mu=0.01,\ \sigma=0.05$. Middle : $\mu=0.1,\ \sigma=0.1$. Right: $\mu=0.1,\ \sigma=0.1$.}
    \label{sys-data}
\end{figure}

\subsection{Numerical experiments with much more complex structures}
Finally, we consider three much more complex examples to verify the effectiveness and robustness of the proposed indicator functions. Here, $\Om$ may have many connected components and many corners. We set $f(z)=|z|+5$ for $z\in\Om$.

Figure \ref{Complex_L-31_lambda-30} shows that both the indicator functions $\mathcal{I}^{\pm}$ with small $L=31$ and $\Lambda=30$ work well for these complex examples. We emphasize again that $L=31$ is much less than needed in the unique proof.
In particular, as has been shown in the previous examples, the indicator function $\mathcal{I}^+$ presents its strong power to determine the source function $f$.


\begin{figure}[htbp]
    \centering
\begin{tabular}{ccc}
        \subfigure{
            \label{FOfComplex-L_31-Lambda_30}
            \includegraphics[width=0.3\textwidth]{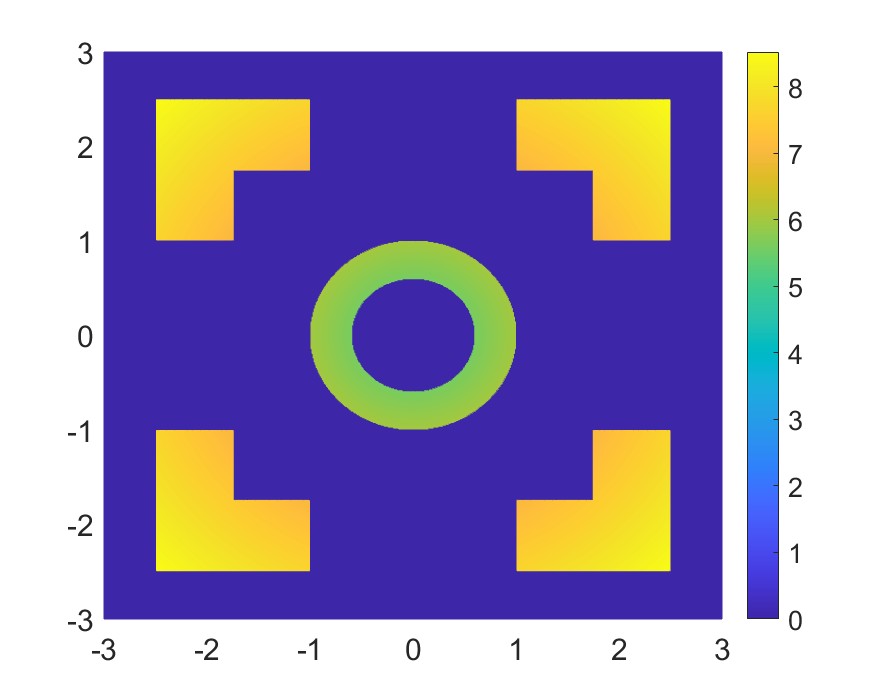}
        }&
        \subfigure{
            \label{MOfComplex-L_31-Lambda_30}
            \includegraphics[width=0.3\textwidth]{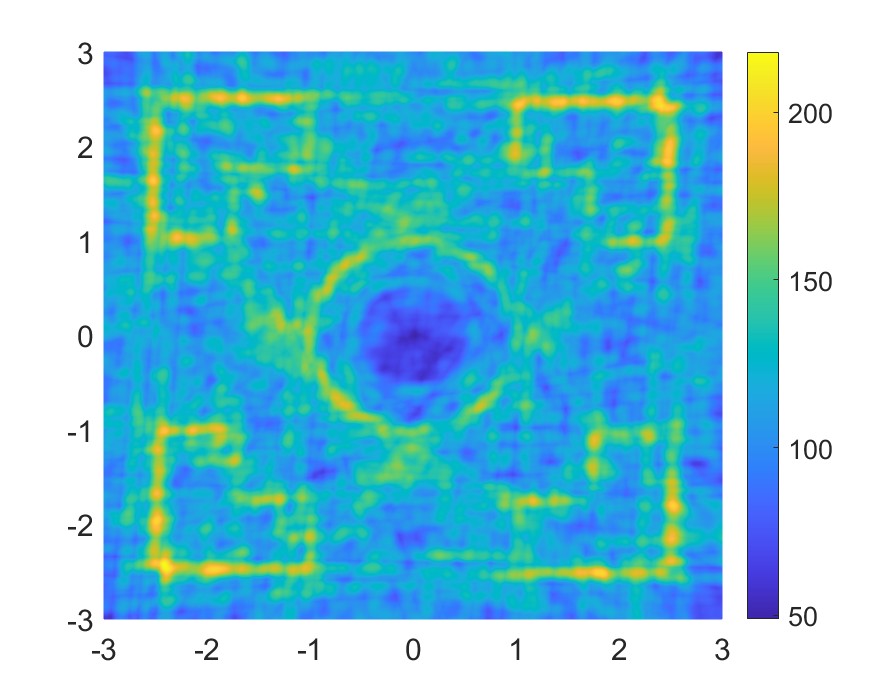}
        }&
         \subfigure{
            \label{POfComplex-L_31-Lambda_30}
            \includegraphics[width=0.3\textwidth]{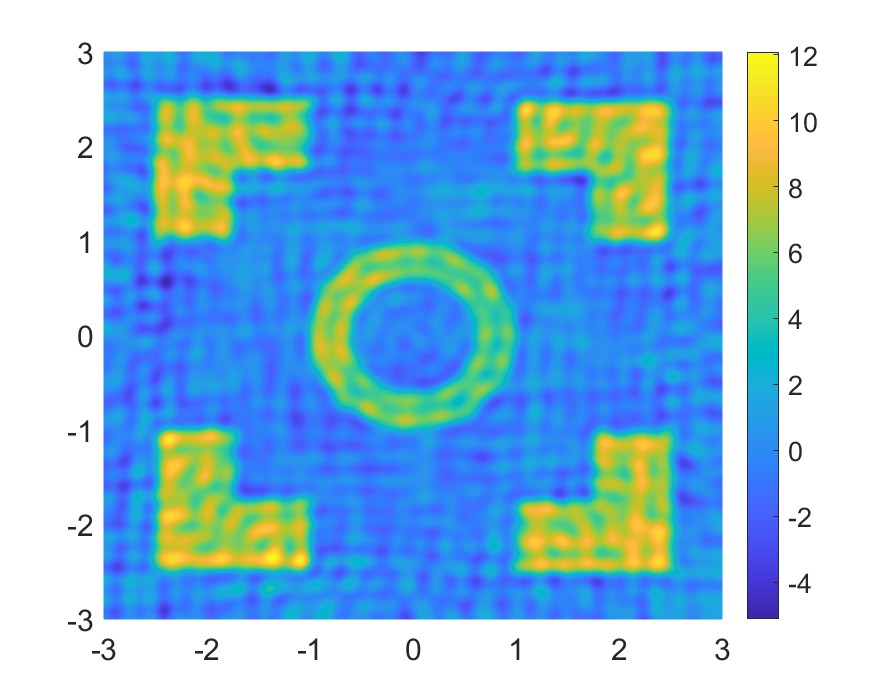}
        }\\
          \subfigure{
            \label{FOfTai-Chi-L_31-Lambda_30}
            \includegraphics[width=0.3\textwidth]{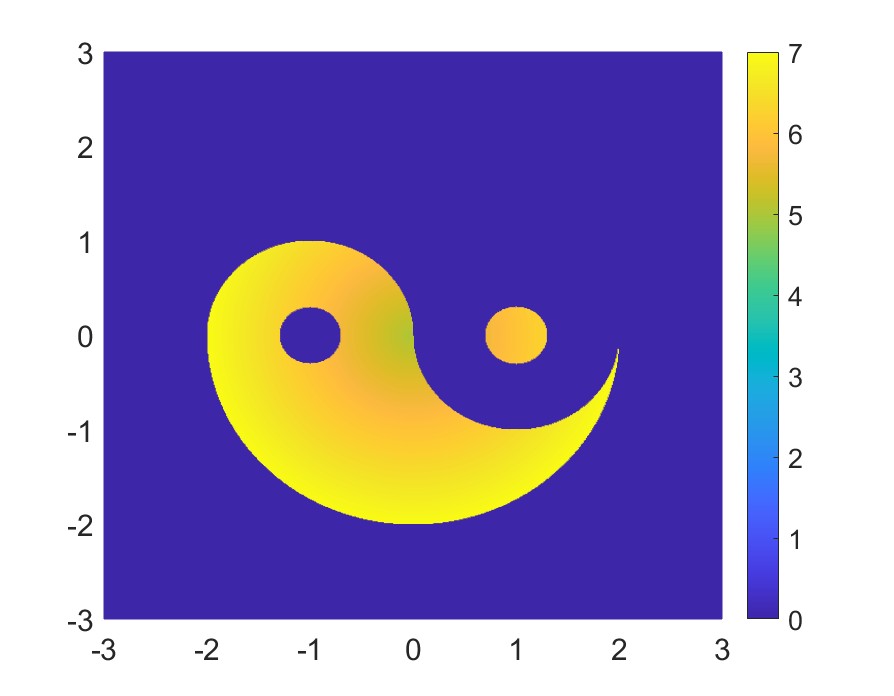}
        }&
        \subfigure{
            \label{MOfTai-Chi-L_31-Lambda_30}
            \includegraphics[width=0.3\textwidth]{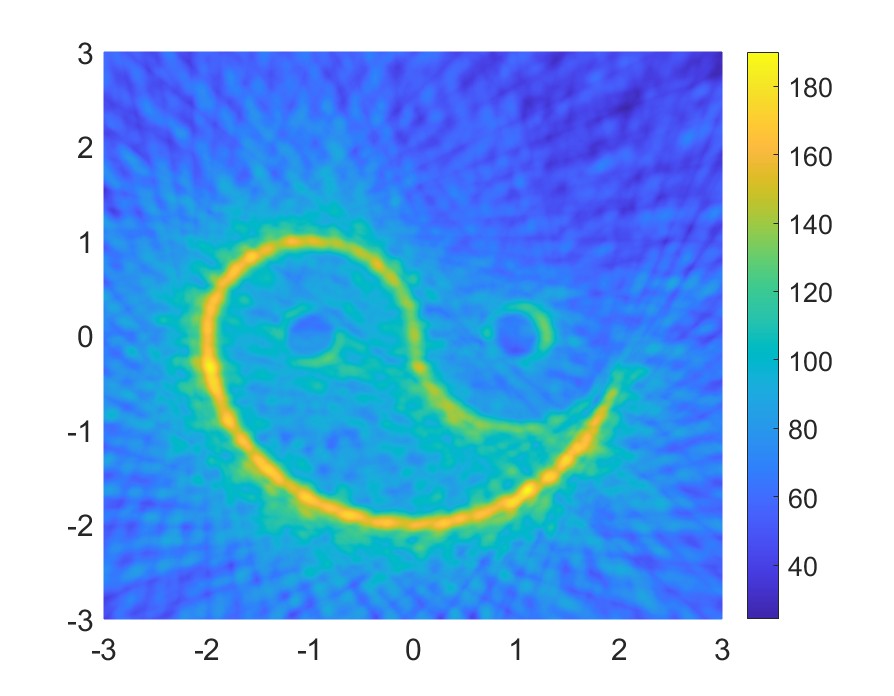}
        }&
         \subfigure{
            \label{POfTai-Chi-L_31-Lambda_30}
            \includegraphics[width=0.3\textwidth]{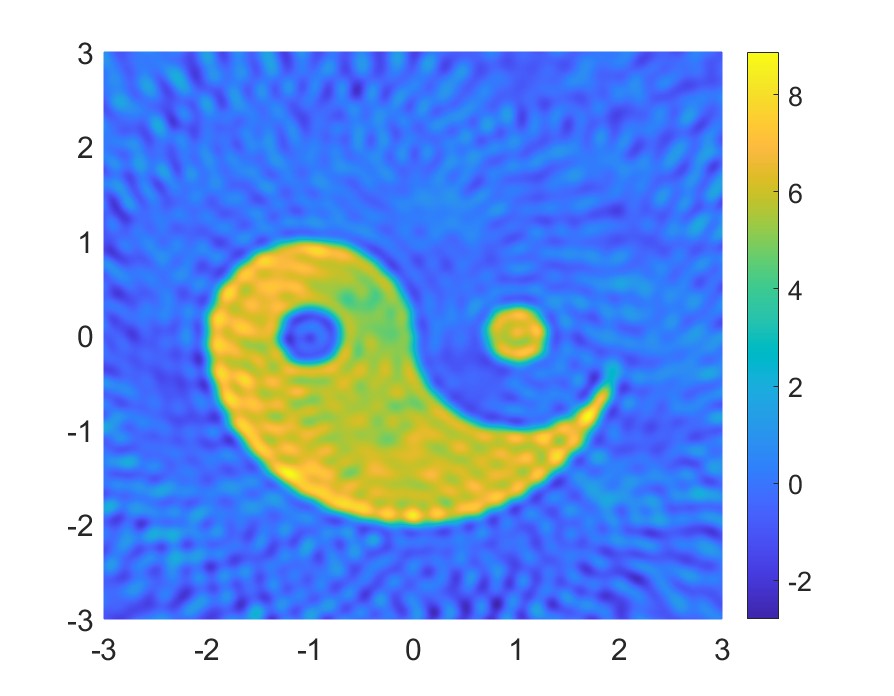}
        }\\
             \subfigure{
            \label{FOfISP-L_31-Lambda_30}
            \includegraphics[width=0.3\textwidth]{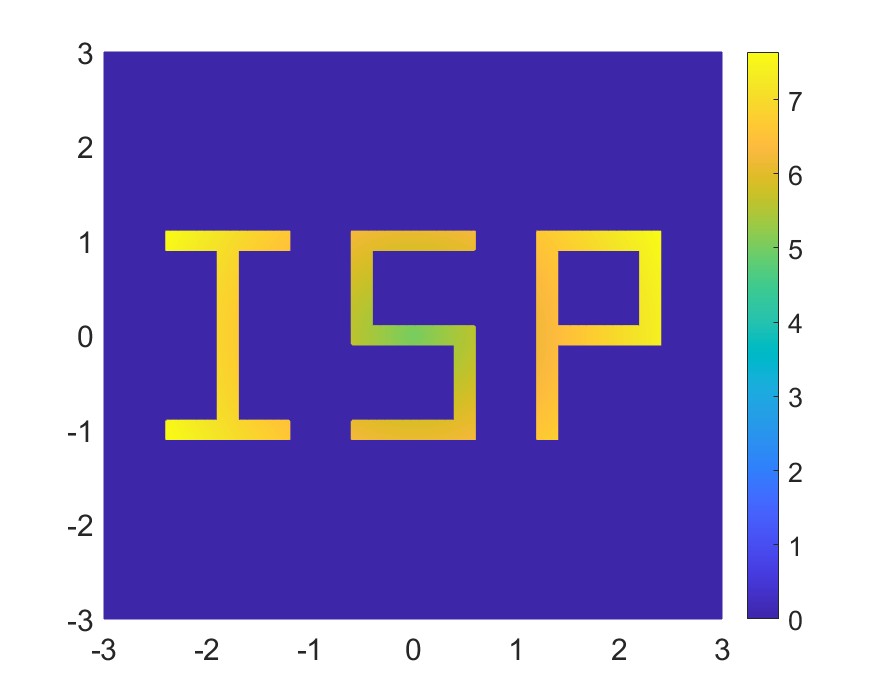}
        }&
        \subfigure{
            \label{MOfISP-L_31-Lambda_30}
            \includegraphics[width=0.3\textwidth]{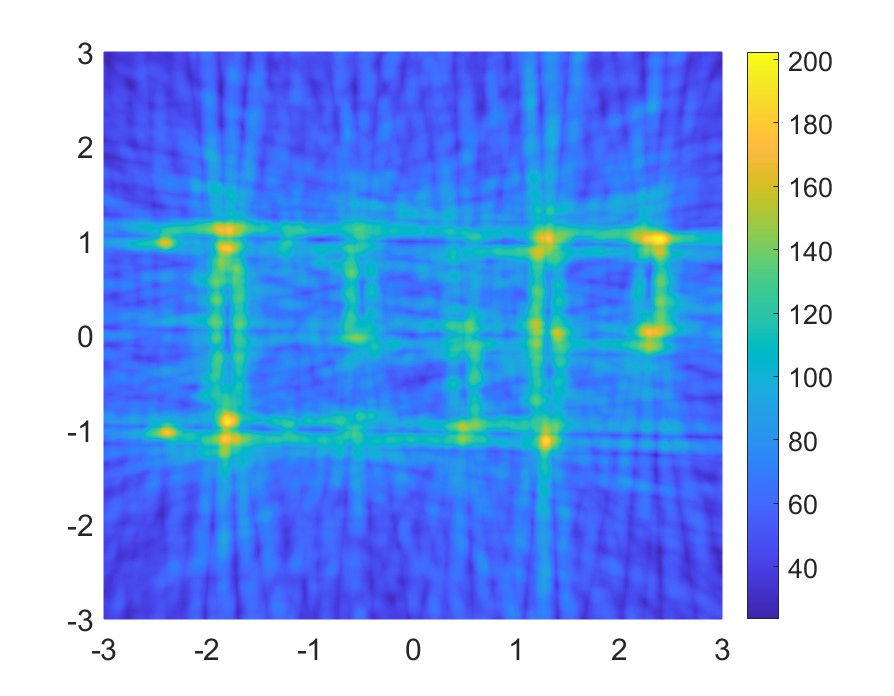}
        }&
         \subfigure{
            \label{PISP-L_31-Lambda_30}
            \includegraphics[width=0.3\textwidth]{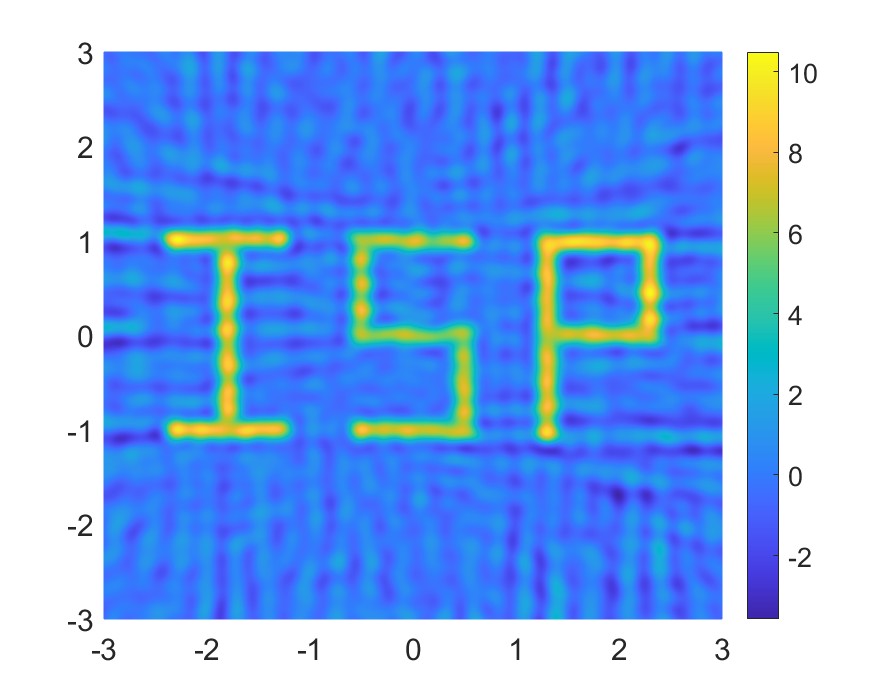}
        }
    \end{tabular}
    \caption{\em Reconstructions of three much more complex structured sources with $L=31,\Lambda=30$. Left: the true source functions. Middle: reconstructions by plotting $\mathcal{I}^-$. Right:  reconstructions by plotting $\mathcal{I}^+$.}
    \label{Complex_L-31_lambda-30}
\end{figure}

To further present the effectiveness of the indicator function $\mathcal{I}^+$  for determining the source $f$, we test the difference of $f(z)$ and $\mathcal{I}^+(z)$ by the following indicator
\be\label{Ieps}
I_{\epsilon}(z):=\left\{
    \begin{array}{ll}
        1, &  {\rm if}\, |f(z)-\mathcal{I}^+(z)|>\epsilon,\\
        0, & {\rm otherwise}.
    \end{array}
    \right.
\en
Considering the $30\%$ relative noise in the far field patterns and the fact $f(z)>5,\ \forall\ z\in\Omega$, we set the threshold $\epsilon=5\times0.3=1.5$. 
Figure \ref{indicator-error} shows that the difference between $f(z)$ and $\mathcal{I}^+(z)$ for $z\in\R^2\ba\pa\Om$ indeed decreases with the increase of $L$. 
Note that the value of the threshold $\epsilon$ is not important. We have tested the other thresholds and founded similar results.
\begin{figure}[htbp]
\centering
    \begin{tabular}{ccc}
        \subfigure[$L=31$.]{
            \label{ErrOfComplex-L_31-Lambda_30}
            \includegraphics[width=0.3\textwidth]{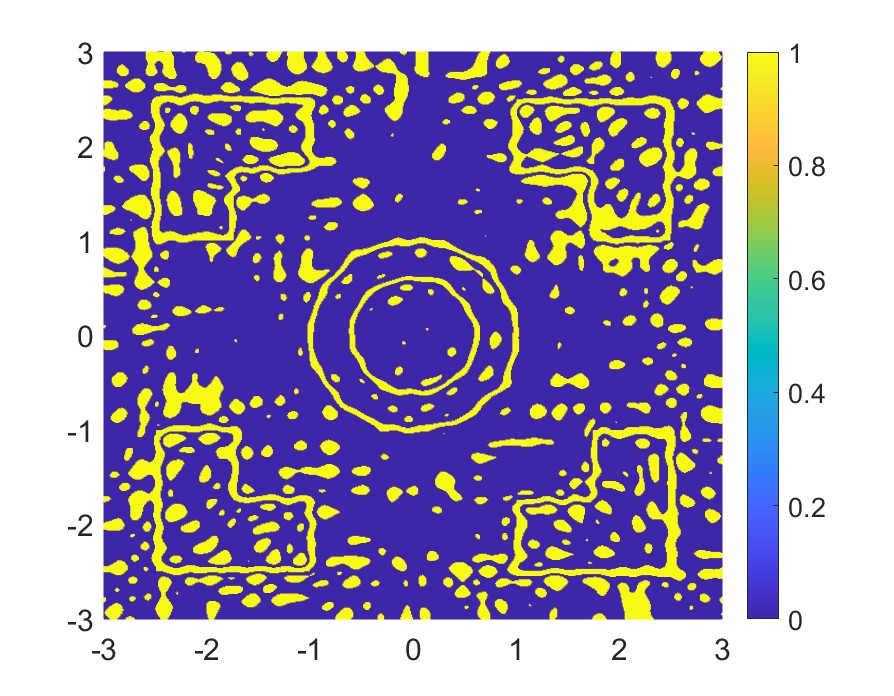}
        }&
        \subfigure[$L=41$.]{
            \label{ErrOfComplex-L_51-Lambda_30}
            \includegraphics[width=0.3\textwidth]{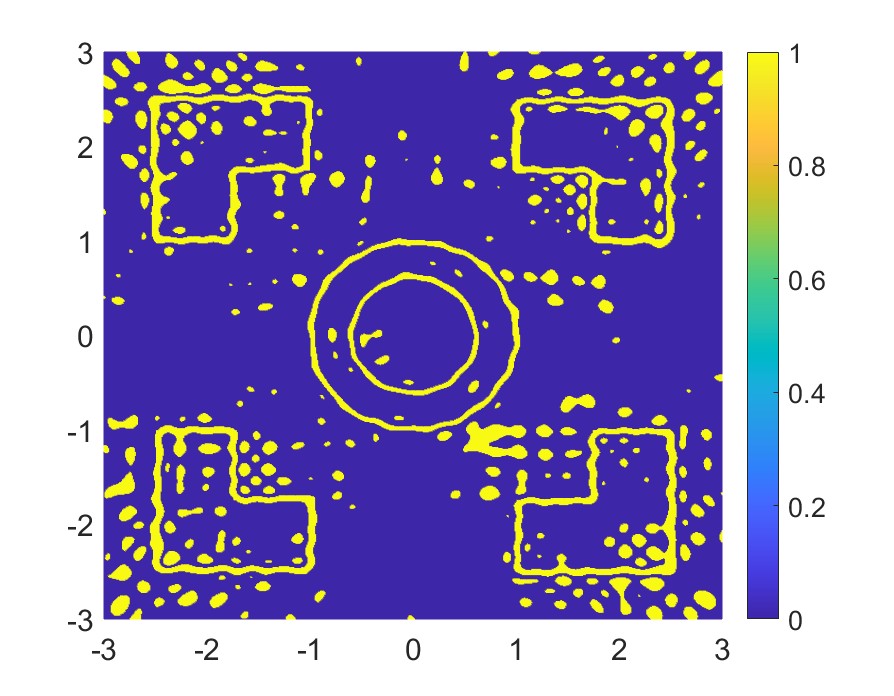}
        }&
         \subfigure[$L=51$.]{
            \label{ErrOfComplex-L_71-Lambda_30}
            \includegraphics[width=0.3\textwidth]{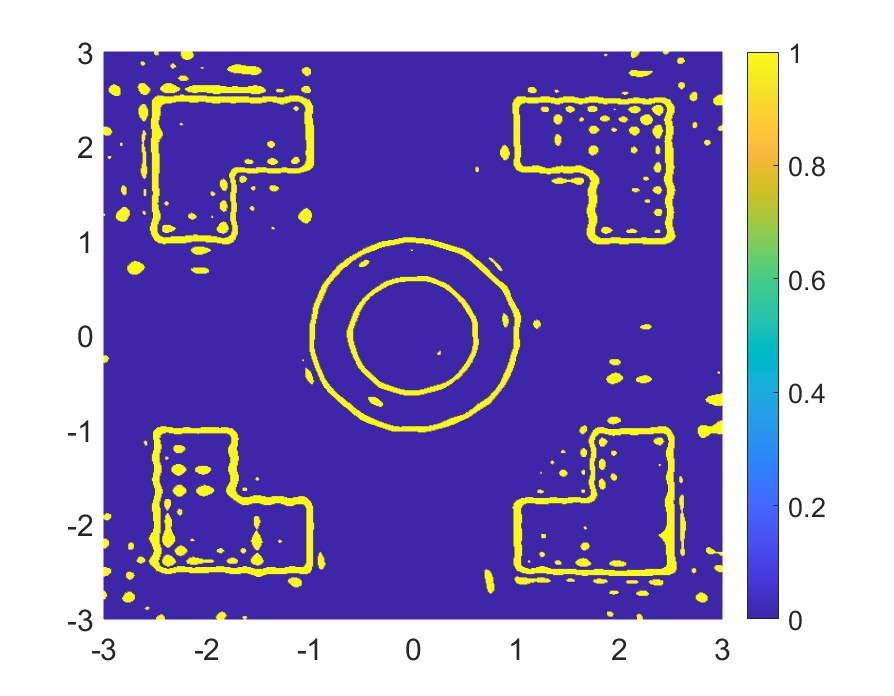}
        }\\
          \subfigure[$L=31$.]{
            \label{ErrOfTai-Chi-L_31-Lambda_30}
            \includegraphics[width=0.3\textwidth]{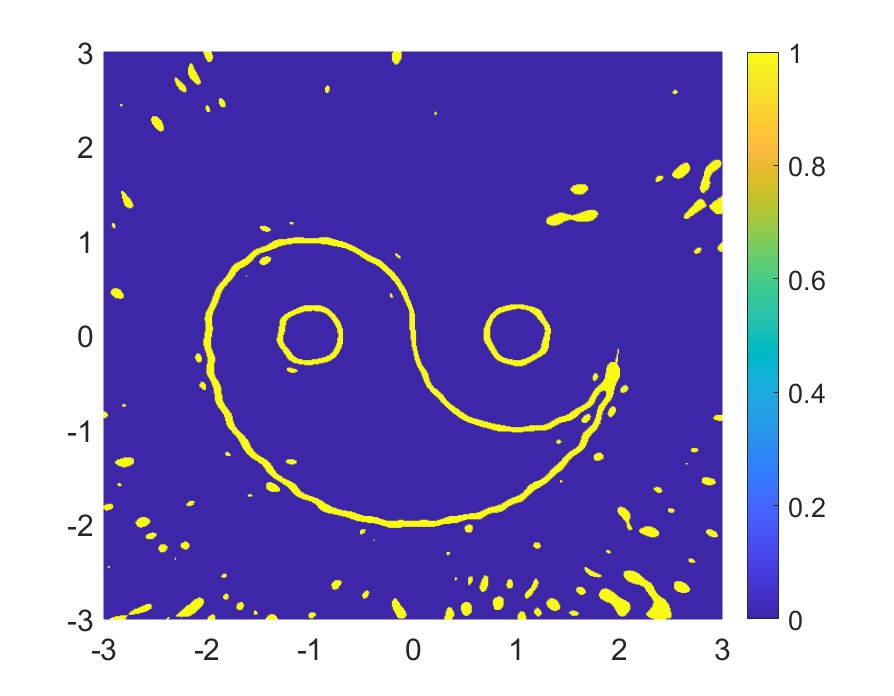}
        }&
        \subfigure[$L=41$.]{
            \label{ErrOfTai-Chi-L_41-Lambda_30}
            \includegraphics[width=0.3\textwidth]{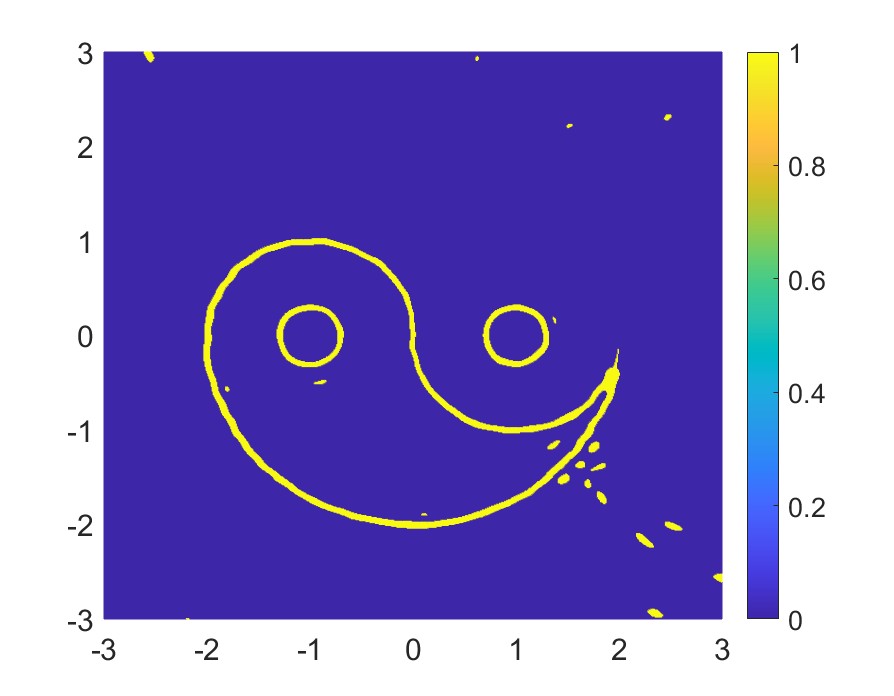}
        }&
         \subfigure[$L=51$.]{
            \label{ErrOfTai-Chi-L_51-Lambda_30}
            \includegraphics[width=0.3\textwidth]{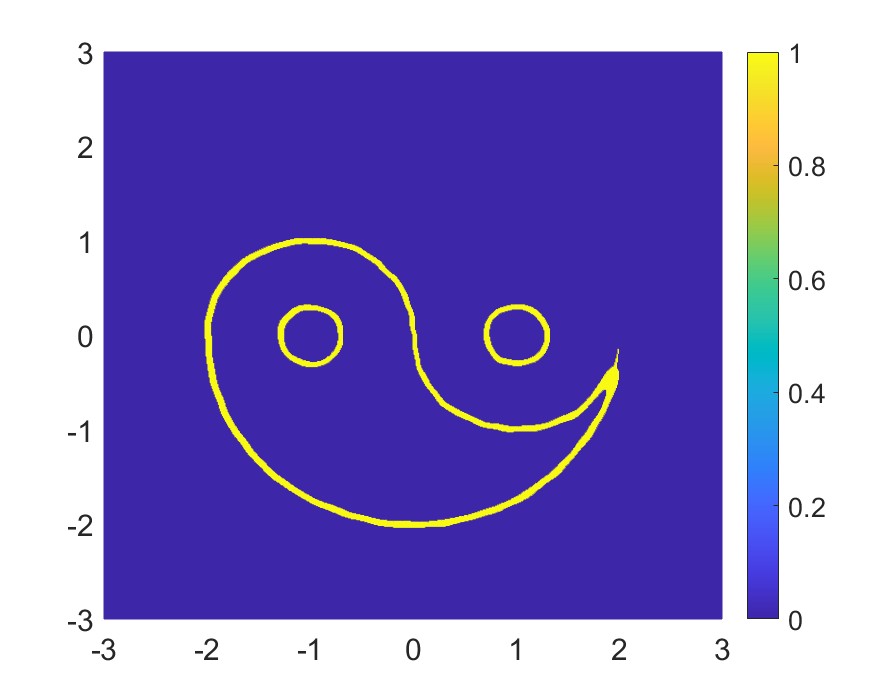}
        }\\
             \subfigure[$L=31$.]{
            \label{ErrOfISP-L_31-Lambda_30}
            \includegraphics[width=0.3\textwidth]{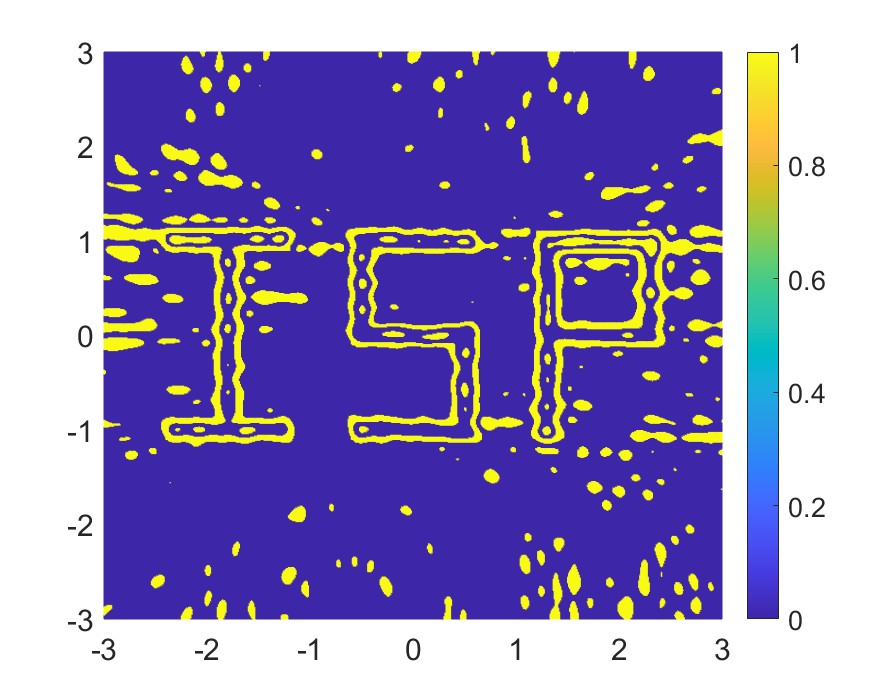}
        }&
        \subfigure[$L=41$.]{
            \label{ErrOfISP-L_41-Lambda_30}
            \includegraphics[width=0.3\textwidth]{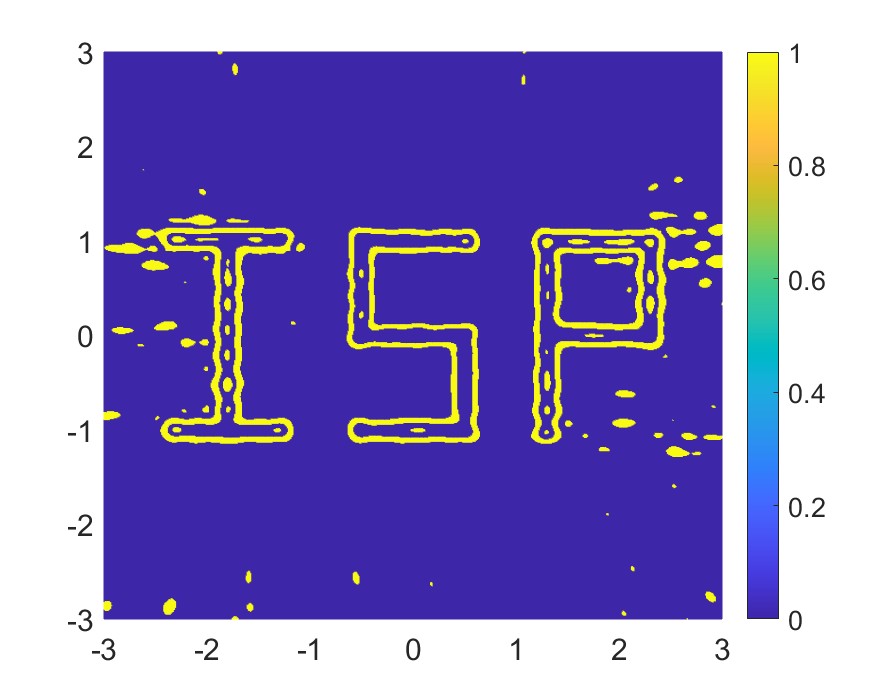}
        }&
         \subfigure[$L=51$.]{
            \label{ErrOfISP-L_51-Lambda_30}
            \includegraphics[width=0.3\textwidth]{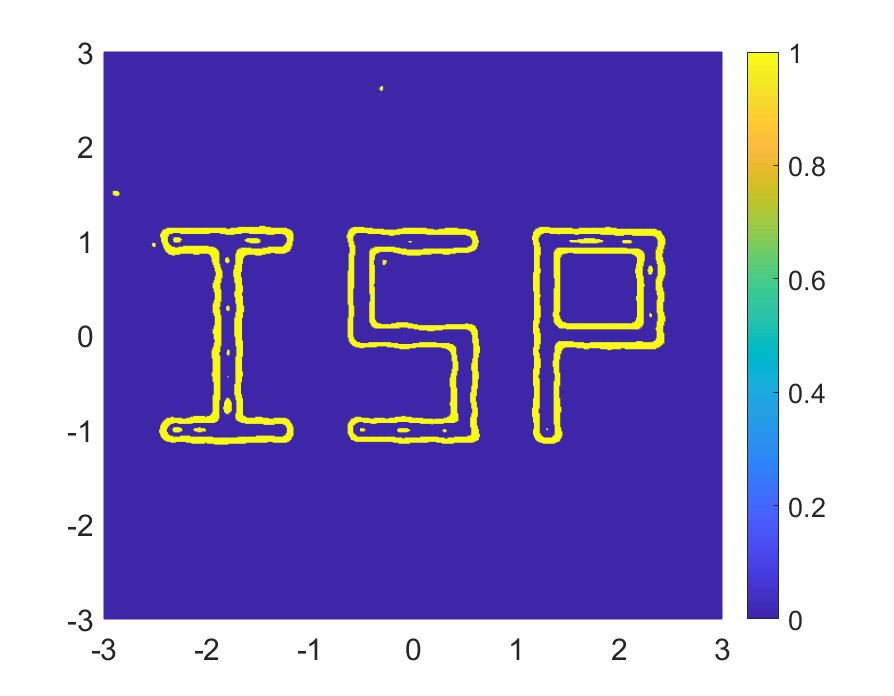}
        }
    \end{tabular}
    \caption{In this example we use the same sources as in Figure \ref{Complex_L-31_lambda-30} but increase the number $L$. We plot the function $I_\epsilon$ defined in \eqref{Ieps}.}
    \label{indicator-error}
\end{figure}
 
\section{Conclusions}
In this work,  we introduce some uniqueness results and two new direct sampling methods for inverse source problem with multi-frequency sparse far field patterns.

For the uniqueness results, we consider the source $f$ that satisfies certain conditions whose support $\Omega$ is composed of some annuluses and some polygons. We first show that the Radon transform $Rf$ of source can be obtained by data. Then, with the help of the relationship between the singularities
in source $f$ and singularities in $Rf$, we prove
that the $\Omega$ as well as the value of $f$ on $\Omega$'s corners can be uniquely determined by 
multi-frequency sparse far field patterns.

For two direct sampling methods, we summarize the main features of the proposed direct sampling methods as follows:
\begin{itemize}
  \item We have used the multi-frequency far field patterns at sparse observation directions; This is important because, in many cases of practical applications, the measurements are only available at finitely many sensors.
  \item The proposed direct sampling methods inherit many advantages of the sampling methods, e.g., very simple and fast to implement, highly robust to noise and making no use of the topological properties of the unknown objects. Besides, the novel sampling methods are able to produce high resolution reconstructions of very complex source supports.
  \item The first indicator function is designed for looking for the boundary of the source support. Such an indicator function is motivated by the uniqueness arguments for polygons and annuluses. To the author's knowledge, this is the first uniqueness result for the source support from the multi-frequency measurements at sparse sensors. 
  \item Generally speaking, the sampling type methods are qualitative methods in the sense that one can only reconstruct the support but not the parameters of the unknown objects. The second indicator function is a modification of the first one. However, we show that it can not only reconstruct the source support but also determine the source function.
\end{itemize}

In conclusion, this work contributes to the field of inverse source problems by providing  the first uniqueness result for the source support $\Omega$ from multi-frequency measurements at sparse sensors and two effective direct sampling methods. Future research will also explore more complex scenarios involving inverse obstacle and medium scattering problems.

\section*{Acknowledgement}
The research of X. Liu is supported by the National Key R\&D Program of China under grant 2024YFA1012303 and the NNSF of China under grant 12371430. The authors thank the referees for their invaluable comments and suggestions which helped improve the paper.

\bibliographystyle{SIAM}

\end{document}